\documentclass{amsart}
\usepackage{amsfonts,amsthm,amsmath}
\usepackage{amssymb}
\usepackage{amscd}
\usepackage[utf8]{inputenc}
\usepackage[a4paper]{geometry}
\usepackage{enumerate}
\usepackage[usenames,dvipsnames]{color}
\usepackage{array}
\usepackage{array,tabularx}

\usepackage{amsthm}

\DeclareMathOperator{\Coeff}{Coeff}
\DeclareMathOperator*{\Res}{Res}
\newcommand\note[1]{\mbox{}\marginpar{ \scriptsize\raggedright
\hspace{1pt}\color{red} #1}}

\numberwithin{equation}{section}
\numberwithin{equation}{subsection}

\theoremstyle{plain}
\newtheorem*{theorem*}{Theorem}

\newtheorem{theorem}[equation]{Theorem}
\newtheorem{lemma}[equation]{Lemma}
\newtheorem{proposition}[equation]{Proposition}

\newtheorem{corollary}[equation]{Corollary}

\newtheorem{thm}[equation]{Theorem}

\theoremstyle{definition}
\newtheorem{example}[equation]{Example}
\newtheorem{remark}[equation]{Remark}

\newtheorem{definition}[equation]{Definition}

\newtheorem{defn}[equation]{Definition}

\newtheorem{bekezdes}[equation]{}

\def\C{\mathbb C}
\def\Q{\mathbb Q}

\def\Z{\mathbb Z}

\def\im{{\rm Im}}

\newcommand{\caly}{{\mathcal Y}}
\newcommand{\calz}{{\mathcal Z}}
\newcommand{\calv}{{\mathcal V}}

\newcommand{\calt}{{\mathcal T}}

\newcommand{\cali}{{\mathcal I}}

\newcommand{\calO}{{\mathcal O}}

\newcommand{\calS}{{\mathcal S}}

\newcommand{\calL}{\mathcal{L}}

\newcommand{\tX}{\widetilde{X}}

\newcommand{\reg}{{\rm reg}}
\newcommand{\cX}{{\mathcal X}}
\newcommand{\cO}{{\mathcal O}}
\newcommand{\bP}{{\mathbb P}}
\newcommand*{\linebundle}{\mathcal{L}}
\newcommand{\bC}{{\mathbb C}}
\newcommand{\cF}{{\mathcal F}}

\newcommand{\eca}{{\rm ECa}}
\newcommand{\pic}{{\rm Pic}}

\newcommand{\m}{\mathfrak{m}}\newcommand{\fr}{\mathfrak{r}}
\newcommand{\mfl}{\mathfrak{L}}

\newcommand{\bt}{{\mathbf t}}

\newcommand{\bZ}{{\mathbb{Z}}}
\newcommand{\bQ}{{\mathbb{Q}}}

\newcommand{\omegl}{\lambda}

\author{J\'anos Nagy}
\address{Rényi Alfréd Institution of Mathematics}
\email{janomo4@gmail.com, janomo@renyi.hu}

\title{Invariants of relatively generic structures on normal surface singularities}

\begin{document}



\begin{abstract}
In the present article we work out a relative setup of generic structures on surface singularities.
We fix an analytic type on a subgraph of a  rational homology sphere resolution graph $\mathcal{T}$ and we choose a relatively generic normal surface singularity $\tX$ with resolution graph $\mathcal{T}$.
We provide formulae for the geometric genus and the analytical Poincaré series of $\tX$. We determine the base point structure of natural line bundles on $\tX$ and give a lower bound on the multiplicity of $\tX$ which is expected to be sharp.
We prove similar results about cohomology numbers of relatively generic line bundles on every singularity with rational homology sphere link.
\end{abstract}

\maketitle

\linespread{1.2}


\pagestyle{myheadings} \markboth{{\normalsize  J. Nagy}} {{\normalsize Line bundles}}


\section{Introduction}\label{s:intr}

For a given negative definite resolution graph $\mathcal{T}$, usually there are several analytically different surface singularities with resolution graph $\mathcal{T}$.
The topology of the link gives restrictions for the analytic invariants of the singularity and there are topological candidates for these invariants for special families of analytic types.
On the other hand, most of the invariants, like the multiplicity, geometric genus or cohomology numbers of line bundles can change if we change the analytic structure.

In \cite{NNA2}, the authors, following the local deformation theory of resolutions of normal surface singularities developed by Laufer\cite{LaDef1}, computed several invariants of generic normal surface singularities with fixed resolution graph, like the geometric genus and analytic Poincaré series or the analytic semigroup. There are however other invariants, like multiplicity or embedded dimension we wish to compute in the case of generic normal surface singularities.

In the present article we work out a relative setup of generic structures on surface singularities. We call a negative definite resolution graph a rational homology sphere resolution graph
if the graph $3$-manifold corresponding to it is a rational homology sphere. We will only deal with rational homology sphere resolution graphs and we will sometimes use the notation 
QHS for them.

We fix an analytic type on a subgraph of a QHS resolution graph $\mathcal{T}$ and we choose a relatively generic normal surface singularity with resolution graph $\mathcal{T}$. We provide formulae for the geometric genus and the analytical Poincaré series of the singularity. We also determine the base point structure of natural line bundles and provide a lower bound for the multiplicity of the singularity, which we excpect to be sharp. 
Similarly assume that we have an arbitrary singularity with resolution $\tX$ with rational homology sphere resolution graph $\mathcal{T}$, and we fix a line bundle on a subresolution and a relatively generic line bundle on the big resolution. We compute the cohomology numbers of this new line bundle.

The answers to these questions are intresting on their own. The main motivation behind these results however is that they give possibilities for inductive proofs of problems
regarding generic surface singularities. We will determine the cohomology numbers of natural line bundles on cycles, or the class of images of Abel maps on generic normal surface singularities in subsequent manuscripts. Parallely in another manuscript with A. Némethi, we determine the multiplicity of generic surface singularities, where one can actually show that the lower bound is indeed sharp (it is in fact the most difficult part).

\bigskip

Consider two QHS resolution graphs $\mathcal{T}_1 \subset \mathcal{T} $ with vertex sets $\calv_1 \subset \calv$, where $\calv = \calv_1 \cup \calv_2$.
Assume that we have a resolution $\tX$ for a normal surface singularity with resolution graph $\mathcal{T}$ and denote its subsingularity $\tX_1$ corresponding to the subgraph $\mathcal{T}_1$. Fix a Chern class $l'$ and  denote the surjective map between the Picard groups $r: \pic^{l'}(\tX) \to \pic^{R_1(l')}(\tX_1)$, where $R_1(l')$ is the cohomological restriction of the Chern class $l'$. With these notations we have the following theorem:

\begin{theorem*}\label{th:hegy2rel}\textbf{A}
Fix an arbitrary line bundle $\mfl\in \pic^{R_1(l')}(\tX_1)$, then for any line bundle $\calL\in r^{-1}(\mfl)$ one has:
\begin{equation*}\label{eq:genericLrel}
\begin{array}{ll}h^1(\tX,\calL)\geq \chi(-l')-
\min_{0\leq l,\ l\in L} \{\,\chi(-l'+l) -
h^1(\tX_1, \mfl(-l))\, \}.\end{array}\end{equation*}
Furthermore, if $\calL$ is generic in $r^{-1}(\mfl)$ then we have equality.
\end{theorem*}

Consider two QHS resolution graphs $\mathcal{T}_1 \subset \mathcal{T} $ with vertex sets $\calv_1 \subset \calv$, where $\calv = \calv_1 \cup \calv_2$.
Assume that we have a surface singularity with resolution $\tX_1$ and resolution graph $\mathcal{T}_1$ and fix cuts $D_{v_2}$ for vertices $v_2 \in \calv_2$ which have got a neighbour in $\calv_1$.

\begin{theorem*}\textbf{B}
Assume that $\tX$ has a relatively generic analytic stucture (defined in section 5) with resolution graph $\mathcal{T}$, corresponding to the analytic structure $\tX_1$ and the cuts $D_{v_2}$ ($E_{v_2} \cap \tX_1 = D_{v_2}$ for all vertices $v_2 \in \calv_2$).

For the geometric genus, $p_g(\tX)$ we have the following formula:

\begin{equation*}
p_g(\tX) = 1 - \min_{E \leq l }(\chi(l)-   h^1(\calO_{\tX_1}(-l) )).                              
\end{equation*}

\end{theorem*}

\medskip

The next theorem determines the base point structure of the line bundle $\calO_{\tX}(-Z_{max})$, where $Z_{max}$ is the maximal ideal cycle of $\tX$ (defined in section 2).

\begin{theorem*} \textbf{C}
Assume that $\tX$ has a relatively generic analytic stucture with resolution graph $\mathcal{T}$, corresponding to the analytic structure $\tX_1$ and the cuts $D_{v_2}$ ($E_{v_2} \cap \tX_1 = D_{v_2}$ for all vertices $v_2 \in \calv_2$).

Assume that the line bundle $\calL = \calO_{\tX_1}(- Z_{max})$ has no base points on $\tX_1$, then:

\medskip

1) The line bundle $\calO_{\tX}(- Z_{max})$ has no base points at the intersection points of exceptional divisors, and it has got only simple base points.

2) If $v \in \calv$, then $\calO_{\tX}(- Z_{max})$ has base point along the exceptional divisor $E_v$ if and only if there exists a cycle $A \geq E_v$, such that:

\begin{equation*}
\chi(Z_{max}) - h^1(\tX_1, \calL) + 1 = \chi(Z_{max} + A) - h^1(\tX_1, \calL( -A)).
\end{equation*}
Furthermore if there is a base point on $E_v$, then there are $-(l_d, E_v)$ distinct base points along the exceptional divisor $E_v$.
\end{theorem*}
\begin{remark}
We will prove Theorem \textbf{C} more generally, in fact we determine the base point structure of all natural line bundles on $\tX$. Notice also that $Z_{max}$ is determined combinatorially
from the resolution graph $\mathcal{T}$ and the analytic invariants of the subsingularity $\tX_1$ in section 7) for a relatively generic singularity.
\end{remark}

Using Theorem \textbf{C} we give a lower bound on the multiplicity of a relatively generic singularity, which we except to be sharp.

\begin{theorem*} \textbf{D}
For a vertex $v \in \calv$ let $t_v$ be the maximal integer, such that there exists a cycle $A \geq t_v \cdot E_v$ for which:
\begin{equation}\label{baseki}
\chi(Z_{max}) - h^1(\tX_1, \calL) + 1 = \chi(Z_{max} + A) - h^1(\tX_1, \calL( -A)). 
\end{equation}

Then we have:

\medskip

1) $Mult(\tX) \geq -Z_{max}^2 - \sum_{v \in \calv} (Z_{max}, E_v) \cdot t_v$.

\medskip

2) Suppose that $t_v = 0$ for all vertices $v \in \calv$, then $Mult(\tX) = -Z_{max}^2$.
\end{theorem*}

\bigskip

\subsection{Motivations for the results:}

In the following we explain briefly why the terminology and the ideas in the present paper is important in the results mentioned above:

The authors in \cite{NNA2} determined the geometric genus of a generic normal surface singularity $\tX$ corresponding to a fixed resolution graph $\mathcal{T}$. The main idea was that if the analytic structure of $\tX$ is generic, then for many Chern classes $l'$, $h^1(\calO_{\tX}(-l'))$ is equal to the $h^1$ of the generic line bundle in $\pic^{l'}(\tX)$, which turns out to be topological.
In many situations with the language of relatively generic analytic structures and relatively generic line bundles we will be able to realise the natural line bundles on generic singularities as (relatively) generic line bundles. It means that we get a more precise formulation of the philosophical idea that on generic singularities $\tX$ the natural line bundles behave like generic line bundles in the corresponding Picard groups (which are much easier to investigate).

For another motivation, fix a singularity with resolution $\tX_1$ and with a QHS  resolution graph $\mathcal{T}_1$. Consider a QHS resolution graph $\mathcal{T}$, such that $\mathcal{T}_1 \subset \mathcal{T}$ and a relative generic singularity with resolution $\tX$ corresponding to $\mathcal{T}$ and $\tX_1$. We get many of the analytic invariants of the relatively generic singularity $\tX$ corresponding to $\mathcal{T}$ and $\tX_1$ by a combinatorial way from the combinatorics of the resolution graph $\mathcal{T}$ and the analytic invariants of $\tX_1$. However the analytic invariants of the singularity $\tX$ are restricted and connected to each other, so one hopes to get information and more restrctions on the possible values of the analytic invariants of $\tX_1$ combinatorially this way, which is a topic of further investigation.

\subsection{Structure of the article:}

In section 2) we summarise the necessary topological and analytic invariants of normal surface singularities.
In section 3) we recall the notations and necessary results from \cite{NNA1} about effective Cartier divisors, Abel maps and differential forms.
In section 4) we recall the results of Laufer regarding deformations of the analytic structure on a resolution space of a normal surface singularity with fixed resolution graph.
In section 5) we define the analouges of the space of effective Cartier divisors, Abel maps and generic line bundles in the relative case, and we prove our main Theorem \textbf{A} about the cohomology numbers of relatively generic line bundles. 
In section 6) we clarify what we mean on a relatively generic analytic structure on a resolution graph $\mathcal{T}$ with respect to a fixed analytic structure $\tX_1$ on a subgraph 
$\mathcal{T}_1$ and cuts $D_{v_2}$ on $\tX_1$. 
In section 7) we provide formulae for the cohomology numbers of natural line bundles on cycles of relatively generic singularities. We determine the analytic Poincaré series and the analytic semigroup and we prove Theorem \textbf{B}.
In section 8) we determine the base points of natural line bundles on relatively generic singularities and we prove Theorem \textbf{C}.
In section 9) we prove Theorem \textbf{D}.
In section 10) as a consequence of our main results we prove a few smaller corollaries.

\textbf{Acknowledgements.} I would like to thank András Stipsicz and András Némethi for many helpful comments and discussions.
The author was supported by OTKA Elvonal KKP126683. 

\section{Preliminaries}\label{s:prel}

\subsection{The resolution}\label{ss:notation}
Let $(X,o)$ be a germ of a complex analytic normal surface singularity,
 and let us fix  a good resolution  $\phi:\widetilde{X}\to X$ of $(X,o)$.
We denote the exceptional curve $\phi^{-1}(0)$ by $E$, and let $\cup_{v\in\calv}E_v$ be
its irreducible components. Set also $E_I:=\sum_{v\in I}E_v$ for any subset $I\subset \calv$.
Let us denote the support of a cycle $l=\sum n_vE_v$ by $|l|=\cup_{n_v\not=0}E_v$.
For more details see \cite{Nfive}.
\subsection{Topological invariants}\label{ss:topol}
Let $\calt$ be the dual resolution graph
associated with $\phi$;  it  is a connected graph.
Then $M:=\partial \widetilde{X}$ can be identified with the link of $(X,o)$, it is also
an oriented  plumbed 3--manifold associated with $\calt$.
We will assume that  $M$ is a rational homology sphere,
or, equivalently,  $\mathcal{T}$ is a tree and all genus
decorations of $\mathcal{T}$ are zero. We use the same
notation $\mathcal{V}$ for the set of vertices, and $\delta_v$ for the valency of a vertex $v$.

The homology group $L:=H_2(\widetilde{X},\mathbb{Z})$, endowed
with a negative definite intersection form  $I=(\,,\,)$, is a lattice. It is
freely generated by the classes of oriented 2--spheres $\{E_v\}_{v\in\mathcal{V}}$.
 The dual lattice $L':=H^2(\widetilde{X},\mathbb{Z})$ is generated
by the (anti)dual classes $\{E^*_v\}_{v\in\mathcal{V}}$ defined
by $(E^{*}_{v},E_{w})=-\delta_{vw}$, the opposite of the Kronecker symbol.
The intersection form embeds $L$ into $L'$, then $H_1(M,\mathbb{Z})\simeq L'/L$, denoted by $H$.
Usually one also identifies $L'$ with those rational cycles $l'\in L\otimes \Q$ for which
$(l',L)\in\Z$, or, $L'={\rm Hom}_\Z(L,\Z)$.


All the $E_v$--coordinates of any $E^*_u$ are strict positive.
We define the Lipman cone as $\calS':=\{l'\in L'\,:\, (l', E_v)\leq 0 \ \mbox{for all $v$}\}$.
It is generated over $\bZ_{\geq 0}$ by $\{E^*_v\}_v$.

\subsection{Analytic invariants}\label{ss:analinv}
The group ${\rm Pic}(\widetilde{X})$
of  isomorphism classes of analytic line bundles on $\widetilde{X}$ appears in the (exponential) exact sequence
\begin{equation}\label{eq:PIC}
0\to {\rm Pic}^0(\widetilde{X})\to {\rm Pic}(\widetilde{X})\stackrel{c_1}
{\longrightarrow} L'\to 0, \end{equation}
where  $c_1$ denotes the first Chern class. Here
$ {\rm Pic}^0(\widetilde{X})=H^1(\widetilde{X},\calO_{\widetilde{X}})\simeq
\C^{p_g}$, where $p_g$ is the {\it geometric genus} of
$(X,o)$. $(X,o)$ is called {\it rational} if $p_g(X,o)=0$.
 Artin in \cite{Artin62,Artin66} characterized rationality topologically
via the graphs; such graphs are called `rational'. By this criterion, $\Gamma$
is rational if and only if $\chi(l)\geq 1$ for any effective non--zero cycle $l\in L_{>0}$.
Here $\chi(l)=-(l,l-Z_K)/2$, where $Z_K\in L'$ is the anticanonical cycle
identified by adjunction formulae
$(-Z_K+E_v,E_v)+2=0$ for all $v$.

The epimorphism
$c_1$ admits a unique group homomorphism section $l'\mapsto s(l')\in {\rm Pic}(\widetilde{X})$,
 which extends the natural
section $l\mapsto \calO_{\widetilde{X}}(l)$ for integral cycles $l\in L$, in a way
such that $c_1(s(l'))=l'$  \cite{Nfive}.
We call $s(l')$ the  {\it natural line bundles} on $\widetilde{X}$.
By  the very  definition, $\calL$ is natural if and only if some power $\calL^{\otimes n}$
of it has the form $\calO_{\tX}(-l)$ for some $l\in L$.

\bekezdes $\mathbf{{Pic}(Z)}.$ \
Similarly, if $Z\in L_{>0}$ is a non--zero effective integral cycle such that its support is $|Z| =E$,
and $\calO_Z^*$ denotes
the sheaf of units of $\calO_Z$, then ${\rm Pic}(Z)=H^1(Z,\calO_Z^*)$ is the group of isomorphism classes
of invertible sheaves on $Z$. It appears in the exact sequence
  \begin{equation}\label{eq:PICZ}
0\to {\rm Pic}^0(Z)\to {\rm Pic}(Z)\stackrel{c_1}
{\longrightarrow} L'\to 0, \end{equation}
where ${\rm Pic}^0(Z)=H^1(Z,\calO_Z)$.

If $Z_2\geq Z_1$ then there are natural restriction maps,
${\rm Pic}(\widetilde{X})\to {\rm Pic}(Z_2)\to {\rm Pic}(Z_1)$, similar restrictions are defined at  ${\rm Pic}^0$ level too.
These restrictions are homomorphisms of the exact sequences  (\ref{eq:PIC}) and (\ref{eq:PICZ}).

Furthermore, we define a section of (\ref{eq:PICZ}) by
$s_Z(l'):=
{\mathcal O}_{\widetilde{X}}(l')|_{Z}$, this also satisfies $c_1\circ s_Z={\rm id}_{L'}$. We write  ${\mathcal O}_{Z}(l')$ for $s_Z(l')$, and we call them
 {\it natural line bundles } on $Z$.

We also use the notations ${\rm Pic}^{l'}(\widetilde{X}):=c_1^{-1}(l')
\subset {\rm Pic}(\widetilde{X})$ and
${\rm Pic}^{l'}(Z):=c_1^{-1}(l')\subset{\rm Pic}(Z)$
respectively. Multiplication by $\calO_{\widetilde{X}}(-l')$, or by
$\calO_Z(-l')$, provide natural affine--space isomorphisms
${\rm Pic}^{l'}(\widetilde{X})\to {\rm Pic}^0(\widetilde{X})$ and
${\rm Pic}^{l'}(Z)\to {\rm Pic}^0(Z)$.

\bekezdes\label{bek:restrnlb} {\bf Restricted natural line bundles.}
The following warning is appropriate.
Let $\tX_1$ be a connected small tubular neighbourhood
of the union of some of the exceptional divisors (hence $\tX_1$ also stays as the resolution
of the singularity obtained by contraction of that union of exceptional  curves). One can repeat the definition of
natural line bundles at the level of $\tX_1$ as well (as a splitting of (\ref{eq:PIC}) applied for
$\tX_1$). However, the restriction to
$\tX_1$ of a natural line bundle of $\tX$ is usually not natural on $\tX_1$:
$\calO_{\tX}(l')|_{\tX_1}\not= \calO_{\tX_1}(R(l'))$
 (where $R:H^2(\tX,\Z)\to H^2(\tX_1,\Z)$ is the natural cohomological 
 restriction), though their Chern classes coincide.

Therefore, in inductive procedure when such restriction is needed,
 we will deal with the family of {\it restricted natural line bundles}, this means the following:

If we have two resolution spaces $\tX_1 \subset \tX$ with resolution graphs $\mathcal{T}_1 \subset \mathcal{T}$ and we have a Chern class $l' \in L'$, then we denote 
by $\calO_{\tX_1}(l') = \calO_{\tX}(l') | \tX_1$ the restriction of the natural line bundle $\calO_{\tX}(l')$.
Similarly if $Z$ is an effective integer cycle on $\tX$ with maybe $|Z| \neq E$, then we denote $\calO_{Z}(l') = \calO_{\tX}(l') | Z$.

Furthermore if $\calL$ is a line bundle on $\tX_1$, then we denote $\calL(l') = \calL \otimes \calO_{\tX}(l')$.
Similarly if $Z$ is  an effective integer cycle on $\tX$ and $\calL$ is a line bundle on the cycle $Z$, then we denote $\calL(l') = \calL \otimes \calO_Z(l')$.

\bekezdes \label{bek:ansemgr} {\bf The analytic semigroups.} \
By definition, the analytic semigroup associated with the resolution $\tX$ is

\begin{equation}\label{eq:ansemgr}
\calS'_{an}:= \{l'\in L' \,:\,\calO_{\tX}(-l')\ \mbox{has no  fixed components}\}.
\end{equation}
It is a subsemigroup of $\calS'$. One also sets $\calS_{an}:=\calS_{an}'\cap L$, a subsemigroup
of $\calS$. In fact, $\calS_{an}$ consists of the restrictions   ${\rm div}_E(f)$ of the divisors
${\rm div}(f\circ \phi)$ to $E$, where $f$ runs over $\calO_{X,o}$. Therefore, if $s_1, s_2\in \calS_{an}$, then
${\rm min}\{s_1,s_2\}\in \calS_{an}$ as well (take the generic linear combination of the corresponding functions).
In particular,  for any $l\in L$, there exists a {\it unique} minimal
$s\in \calS_{an}$ with $s\geq l$.

Similarly, for any $h\in H=L'/L$ set $\calS'_{an,h}:\{l'\in \calS_{an}\,:\, [l']=h\}$.
Then for any  $s'_1, s'_2\in \calS_{an,h}$ one has
${\rm min}\{s'_1,s'_2\}\in \calS_{an,h}$, and so
for any $l'\in L'$   there exists a unique minimal
$s'\in \calS_{an,[l']}$ with $s'\geq l'$.

For any $l'\in\calS_{an}'$ there exists an ideal sheaf $\cali(l')$ with 0--dimensional support along $E$ such that
 $H^0(\tX,\calO_{\tX}(-l'))\cdot \calO_{\widetilde{X}}=\calO_{\widetilde{X}}(-l')\cdot \cali(l')$.
The ideal $\cali(l') $ describes the space of base points of the line bundle $\calO_{\tX}(-l')$.
If $l'\in\calS'_{an}$ and  the divisor of a generic global section of $\calO_{\tX}(-l')$ intersects
$E_v$, then $(l',E_v)<0$. In particular, if $p\in E_v$ is a  base point then necessarily $(l',E_v)<0$.

Choose a base point $p$ of $\calO_{\tX}(-l')$, and assume that it is a regular point of $E$, and that $\cali(l')_p$
 in the
local ring $\calO_{\tX,p}$ is $(x^t,y)$, where $x,y$ are some local coordinates at $p$  with $\{x=0\}=E$ (locally),
and $t\geq 1$.
Then we say that $p$ is a {\it $t$--simple base point}. In such cases we write
$t=t(p)$. Furthermore, $p$ is called {\it simple}
if it is $t$--simple for some $t\geq 1$.

Consider a Chern class $l' \in S'_{an}$ and a base point $p \in E_{v, reg}$ of a natural line $\calO_{\tX}(-l')$, which is simple, then there is another interpretation of the positive integer $t$, such that $p$ is $t$-simple.

Assume that we have a generic section $s \in H^0(\calO_{\tX}(-l'))$ and $D = |s|$, then we know that $D$ has a cut $D'$, which is transversal at the base point $p$.
Blow up the exceptional divisor $E_v$ along the cut $D'$ sequentially, so let us blow up first at the point $p$ and let the new exceptional divisor be $E_{v_1}$. Denote
the strict transform of the cut $D'$ with the same notation. After that blow up $E_{v_1}$ at the intersection point $E_{v_1} \cap D'$ and let the new exceptional divisor be $E_{v_2}$ and so on.

Let us denote the given resolution at the $i$-th step by $\tX_i$ with the blow up map $b_i : \tX_i \to \tX$ and consider the natural line bundle $\calL_i = \calO_{\tX_i}(-b_i^*(l') - \sum_{1 \leq j \leq i} j \cdot E_{v_j}) = \calO_{\tX_i}(D_{st})$, where $D_{st}$ is the strict transform of the divisor $D$.
Let $t$ be the minimal number, such that $\calL_t$ hasn't got a base point along the excpetional divisor $E_{v_t}$.
Equivalently $t$ is the maximal integer, such that $H^0(\tX_t, \calL_t) = H^0(\calO_{\tX_t}( - b_t^*(l')))$ and $h^1(\tX_t, \calL_t)  = h^1(\calO_{\tX}(-l')) + t$. 
In this case $p$ is a $t$-simple base point ot the natural line bundle $\calO_{\tX}(-l')$.

\subsection{Notations.} We will write $Z_{min}\in L$ for the  {\it minimal} (or fundamental, or Artin) cycle, which is
the minimal non--zero cycle of $\calS'\cap L$ \cite{Artin62,Artin66}. Yau's {\it maximal ideal cycle}
$Z_{max}\in L$ defines the  divisorial part of the pullback of the maximal ideal $\m_{X,o}\subset \calO_{X,o}$, i.e.
 $\phi^*{\m_{X,o}}\cdot \calO_{\widetilde{X}}=\calO_{\widetilde{X}}(-Z_{max})\cdot \cali$,
where $\cali$ is an ideal sheaf with 0--dimensional support \cite{Yau1}. In general $Z_{min}\leq Z_{max}$.

\section{Effective Cartier divisors and Abel maps}

  In this section we review some needed material from \cite{NNA1}.

We fix a good resolution $\phi:\tX\to X$ of a normal surface singularity, whose link is a rational homology sphere. 

\subsection{} \label{ss:4.1}
Let us fix an effective integral cycle  $Z\in L$, $Z\geq E$. (The restriction $Z\geq E$ is imposed by the
easement of the presentation, everything can be adopted  for $Z>0$).

Let $\eca(Z)$  be the space of effective Cartier (zero dimensional) divisors supported on  $Z$.
Taking the class of a Cartier divisor provides  a map
$c:\eca(Z)\to \pic(Z)$.
Let  $\eca^{l'}(Z)$ be the set of effective Cartier divisors with
Chern class $l'\in L'$, that is,
$\eca^{l'}(Z):=c^{-1}(\pic^{l'}(Z))$.

 \begin{theorem}\cite{NNA1}\label{th:smooth} If $l'\in-\calS'$ then the following facts hold.

  (1)  $\eca^{l'}(Z)$ is a smooth variety of dimension $(l',Z)$.

  (2) The natural restriction  $r:\eca^{l'}(Z)\to \eca^{l'}(E)$ is a
  locally trivial  fiber bundle with fiber isomorphic to an affine space. Hence,
 the homotopy type of $\eca^{l'}(Z)$ is independent of the choice of $Z$ and
 it depends only on the topology of $(X,o)$.
 \end{theorem}

We consider the restriction of $c$, $c^{l'}:\eca^{l'}(Z)
\to \pic^{l'}(Z)$ too, sometimes still denoted by $c$. 

For any $Z_2\geq Z_1>0$ one has the natural  commutative diagram
\begin{equation}\label{eq:diagr}
\begin{picture}(200,45)(0,0)
\put(50,37){\makebox(0,0)[l]{$
\eca^{l'}(Z_2)\,\longrightarrow \, \pic^{l'}(Z_2)$}}
\put(50,8){\makebox(0,0)[l]{$
\eca^{l'}(Z_1)\,\longrightarrow \, \pic^{l'}(Z_1)$}}
\put(70,22){\makebox(0,0){$\downarrow$}}
\put(135,22){\makebox(0,0){$\downarrow$}}
\end{picture}
\end{equation}

As usual, we say that $\calL\in \pic^{l'}(Z)$ has no fixed components if
\begin{equation}\label{eq:H_0}
H^0(Z,\calL)_{\reg}:=H^0(Z,\calL)\setminus \bigcup_v H^0(Z-E_v, \calL(-E_v))
\end{equation}
is non--empty.
Note that $H^0(Z,\calL)$ is a module over the algebra
$H^0(\calO_Z)$, hence one has a natural action of $H^0(\calO_Z^*)$ on
$H^0(Z, \calL)_{\reg}$. This second action is algebraic and free. 
 Furthermore, $\calL\in \pic^{l'}(Z)$ is in the image of $c$ if and only if
$H^0(Z,\calL)_{\reg}\not=\emptyset$, in this case, $c^{-1}(\calL)=H^0(Z,\calL)_{reg}/H^0(\calO_Z^*)$.

One verifies that $\eca^{l'}(Z)\not=\emptyset$ if and only if $-l'\in \calS'\setminus \{0\}$. Therefore, it is convenient to modify the definition of $\eca$ in the case $l'=0$: we (re)define $\eca^0(Z)=\{\emptyset\}$,
as the one--element set consisting of the `empty divisor'. We also take $c^0(\emptyset):=\calO_Z$. Then we have
\begin{equation}\label{eq:empty}
\eca^{l'}(Z)\not =\emptyset \ \ \Leftrightarrow \ \ l'\in -\calS'.
\end{equation}
If $l'\in -\calS'$  then
  $\eca^{l'}(Z)$ is a smooth variety of dimension $(l',Z)$. Moreover,
if $\calL\in \im (c^{l'}(Z))$ (the image of $c^{l'}$)
then  the fiber $c^{-1}(\calL)$
 is a smooth, irreducible quasiprojective variety of  dimension
 \begin{equation}\label{eq:dimfiber}
\dim(c^{-1}(\calL))= h^0(Z,\calL)-h^0(\calO_Z)=
 (l',Z)+h^1(Z,\calL)-h^1(\calO_Z).
 \end{equation}

\bekezdes \label{bek:I}
Consider again  a Chern class $l'\in-\calS'$ as above.
The $E^*$--support $I(l')\subset \calv$ of $l'$ is defined via the identity  $l'=\sum_{v\in I(l')}a_vE^*_v$ with all
$\{a_v\}_{v\in I}$ nonzero. Its role is the following.

Besides the Abel map $c^{l'}(Z)$ one can consider its `multiples' $\{c^{nl'}(Z)\}_{n\geq 1}$ as well. It turns out
(cf. \cite[\S 6]{NNA1})   that $n\mapsto \dim \im (c^{nl'}(Z))$
is a non-decreasing sequence, and   $\im (c^{nl'}(Z))$ is an affine subspace
for $n\gg 1$, whose dimension $e_Z(l')$ is independent of $n\gg 0$, and essentially it depends only
on $I(l')$.
We denote the linearisation of this affine subspace by $V_Z(I) \subset H^1(\calO_Z)$, or if the cycle $Z \gg 0$, then $ V_{\tX}(I) \subset H^1(\calO_{\tX})$.

Moreover, by \cite[Theorem 6.1.9]{NNA1},
\begin{equation}\label{eq:ezl}
e_Z(l')=h^1(\calO_Z)-h^1(\calO_{Z|_{\calv\setminus I(l')}}),
\end{equation}
where $Z|_{\calv\setminus I(l')}$ is the restriction of the cycle $Z$ to its $\{E_v\}_{v\in \calv\setminus I(l')}$
coordinates.

If $Z\gg 0$ (i.e. all its $E_v$--coordinated are very large), then (\ref{eq:ezl}) reads as
\begin{equation}\label{eq:ezlb}
e_Z(l')=h^1(\calO_{\tX})-h^1(\calO_{\tX(\calv\setminus I(l'))}),
\end{equation}
where $\tX(\calv\setminus I(l'))$ is a convenient small neighbourhood of $\cup_{v\in \calv\setminus I(l')}E_v$.

Let $\Omega _{\tX}(I)$ be the subspace of $H^0(\tX\setminus E, \Omega^2_{\tX})/ H^0(\tX,\Omega_{\tX}^2)$ generated by differential forms which have no poles along $E_I\setminus \cup_{v\not\in I}E_v$.
Then, cf. \cite[\S8]{NNA1},
\begin{equation}\label{eq:ezlc}
h^1(\calO_{\tX(\calv\setminus I)})=\dim \Omega_{\tX}(I).
\end{equation}

Similarly let $\Omega _{Z}(I)$ be the subspace of $H^0(\calO_{\tX}(K + Z))/ H^0(\calO_{\tX}(K))$ generated by differential forms which have no poles along $E_I\setminus \cup_{v\not\in I}E_v$.
Then, cf. \cite[\S8]{NNA1},
\begin{equation}\label{eq:ezlc}
h^1(\calO_{Z_{(\calv\setminus I)}})=\dim \Omega_{Z}(I).
\end{equation}

We have also the following duality from \cite{NNA1} supporting the equalities above:

\begin{theorem}\cite{NNA1}\label{th:DUALVO}
Via Laufer duality one has  $V_{\tX}(I)^*=\Omega_{\tX}(I)$ and $V_{Z}(I)^*=\Omega_Z(I)$.
\end{theorem}

\section{Results of Laufer on deformations of normal surface singularities}

In this subsection we recall some results of Laufer regarding deformations of the analytic structure on a resolution space of a normal surface singularity with fixed resolution graph
(and deformations of non--reduced analytic spaces supported on  exceptional curves) \cite{LaDef1}.

First, consider a normal surface singularity $(X,o)$ and a good resolution $\phi:(\tX,E)\to (X,o)$ with reduced exceptional curve $E=\phi^{-1}(o)$, whose irreducible decomposition is $\cup_{v\in\calv}E_v$ and dual graph $\mathcal{T}$.

Let $\cali_v$ be the ideal sheaf of $E_v\subset \tX$, then for arbitrary positive integers $\{r_v\}_{v\in \calv}$ one defines two objects, an analytic one and a topological (combinatorial) one.
 At the analytic level, one sets the ideal sheaf  $\cali(r):=\prod_v \cali_v^{r_v}$
 and the non--reduces space $\calO_{Z(r)}:=\calO_{\tX}/\cali(r)$ supported on $E$.

  The topological object is a graph with multiplicities, denoted by $\mathcal{T}(r)$. As a non--decorated graph it coincides with the resolution graph $\mathcal{T}$ without decorations. Additionally each vertex $v$ has a
  `multiplicity decoration' $r_v$, and we put also the self--intersection  decoration $E_v^2$
  whenever $r_v>1$. (Hence, the vertex $v$ does not inherit the self--intersection decoration
  of $v$ if $r_v=1$).
 Note that the  abstract 1--dimensional analytic space $Z(r)$ determines by its reduced structure
 the shape of the dual graph $\mathcal{T}$, and by its non--reduced structure
 all the multiplicities $\{r_v\}_{v\in\calv}$, and additionally,
 all the self--intersection numbers $E_v^2$ for those $v$'s when  $r_v>1$.

 We say that the space $Z(r)$ has topological type $\mathcal{T}(r)$.

Clearly, the analytic structure of $(X,o)$, hence of $\tX$ too, determines each 1--dimensional
non--reduced space $\calO_{Z(r)}$.  The converse is also true in the following sense.
\begin{theorem}\label{th:La1} \ \cite[Th. 6.20]{Lauferbook},\cite[Prop. 3.8]{LaDef1}
(a) Consider an abstract 1--dimensional space $\calO_{Z(r)}$, whose topological type
$\mathcal{T}(r)$ can be completed to a negative definite graph $\mathcal{T}$ (or, lattice $L$).
Then there exists a 2--dimensional
manifold $\tX$ in which $Z(r)$ can be embedded with support $E$
such that the intersection matrix inherited from the embedding $E\subset \tX$ is the
negative definite lattice $L$.
In particular (since by Grauert theorem \cite{GRa}
the exceptional locus $E$ in $\tX$ can be contracted to a normal singularity),
any such $Z(r)$ is always associated with a normal surface singularity (as above).

(b)  Suppose that we have two singularities $(X,o)$ and $(X',o)$ with good resolutions as above with the
same resolution graph $\mathcal{T}$. Depending solely on $\mathcal{T}$ the integers $\{r_v\}_v$ may be chosen so
large that if $\calO_{Z(r)}\simeq \calO_{Z'(r)}$, then $E\subset \tX$ and $E'\subset \tX'$ have
biholomorphically equivalent neighbourhoods via a map taking $E$ to $E'$.

\end{theorem}
In particular, in the deformation theory of $\tX$ it is enough to consider the deformations of
non--reduced spaces of type $\calO_{Z(r)}$.

Fix a non--reduced 1--dimensional space $Z=Z(r)$ with topological type $\mathcal{T}(r)$ and we also choose a closed subspace $Y$ of $Z$ (whose support can be smaller, it can be even empty).
More precisely, $(Z,Y)$ is locally isomorphic with $(\C\{x,y\}/(x^ay^b),\C\{x,y\}/(x^cy^d))$,
where $a\geq c\geq 0$, $b\geq d\geq 0$.

 The ideal of $Y$ in $\calO_Z$ is denoted by $\cali_Y$.

\begin{definition}\label{def:1}\ \cite[Def. 2.1]{LaDef1}
 A deformation of $Z$, fixing $Y$, consists of the following data:

(i) There  exists an analytic space $\calz$ and a proper map $\omegl:\calz\to Q$, where
$Q$ is a manifold containing a distinguished point $0$.

(ii) Over a point $q\in Q$ the fiber $Z_q$ is the subspace of $\calz$ determined by the ideal
sheaf $\omegl^* (\mathfrak{m}_q)$ (where $\mathfrak{m}_q$ is the maximal ideal of $q$). $Z$
is isomorphic with $Z_0$, usually they are  identified.

(iii) $\omegl$ is a trivial deformation of $Y$ (that is, there is a closed subspace
$\caly\subset \calz$ and the restriction of $\omegl$ to $\caly$ is a trivial deformation of $Y$).

(iv) $\omegl$ is {\it locally trivial} in a way which extends the trivial deformation $\omegl|_{\caly}$.
This means that for any point $q\in Q$ and $z\in \calz$ there exist a neighborhood $W$ of $z$ in $\calz$,
a neighborhood $V$ of $z$ in $Z_q$, a neighborhood $U$ of $q$ in $Q$, and an isomorphism
$\phi:W\to V\times U$ such that $\omegl|_W=pr_2\circ \phi$ (compatibly with the trivialization
of $\caly$ from (iii)), where $pr_2$ is the second projection.
\end{definition}
One verifies that under deformations (with connected base space) the topological type of the fibers
$Z_q$,  namely $\mathcal{T}(r)$, stays constant (see \cite[Lemma 3.1]{LaDef1}).

\begin{definition}\label{def:2}\ \cite[Def. 2.4]{LaDef1}
A deformation $\omegl:\calz\to Q$ of $Z$, fixing $Y$, is complete at $0$ if, given any deformation
$\tau:{\mathcal P}\to R$ of $Z$ fixing $Y$, there is a neighbourhood $R'$ of $0$ in $R$  and a
 holomorphic map $f:R'\to Q$ such that $\tau$ restricted to $\tau^{-1}(R')$ is the deformation
$f^*\omegl$. Furthermore, $\omegl$ is complete if it is complete at each point $q\in Q$.
\end{definition}
Laufer proved the following  results.
\begin{theorem}\label{th:La2}\ \cite[Theorems 2.1, 2.3, 3.4, 3.6]{LaDef1}
Let $\theta_{Z,Y}={\mathcal Hom}_{Z}(\Omega^1_Z,\cali_Y)$ be the sheaf of germs of vector fields on $Z$  which vanish on $Y$, and let $\omegl :\calz\to Q$ be a deformation of $Z$, fixing $Y$.

 (a) If the Kodaira--Spencer map
 $\rho_0:T_0Q\to H^1(Z,\theta_{Z,Y})$ is surjective then $\omegl$ is complete at $0$.

 (b) If $\rho_0$ is surjective than $\rho_q$ is surjective for all $q$ sufficiently near to $0$.

 (c)  There exists a deformation $\omegl$ with $\rho_0$ bijective. In such a case in a neighbourhood $U$ of $0$ the deformation is essentially unique, and  the fiber above $q$ is isomorphic to $Z$
 for only at most countably many $q$ in $U$.
\end{theorem}

\bekezdes\label{bek:funk} {\bf Functoriality} 

Let $Z'$ be a closed subspace of $Z$ such that
$\cali_{Z'}\subset \cali_Y\subset \calO_Z$. Then there is a natural reduction of pairs
$(\calO_Z,\calO_Y)\to (\calO_{Z'},\calO_Y)$. Hence, any deformation $\omegl:\calz\to Q$
of $Z$ fixing $Y$ reduces to a deformation $\omegl':\calz'\to Q$
of $Z'$ fixing $Y$. Furthermore, of $\omegl$ is complete then $\omegl'$ is automatically
complete as well (since $H^1(Z,\theta_{Z,Y})\to H^1(Z',\theta_{Z',Y})$ is onto).

\section{Relatively generic line bundles}

In this section we wish to generalize the results in \cite{NNA1} about cohomology of generic line bundles to the `relative' situation. In the next paragraphs we explain the setup.

\subsection{The relative setup.}\label{ss:relative}

We consider a cycle $Z \geq E$ on a resolution $\tX$ of a normal surface singularity with rational homology sphere resolution graph $\mathcal{T}$.
We also consider a smaller cycle $Z_1 \leq Z$, where we denote $|Z_1| = \calv_1$ and the not nessacarily connected subgraph corresponding to it by $\mathcal{T}_1$.

We have the restriction map $r:\pic(Z)\to \pic(Z_1)$ and one has also the cohomological restriction operator
  $R_1 : L'(\mathcal{T}) \to L_1':=L'(\mathcal{T}_1)$ defined as $R_1(E^*_v(\mathcal{T}))=E^*_v(\mathcal{T}_1)$ if $v\in \calv_1$, and
$R_1(E^*_v(\mathcal{T}))=0$ otherwise. For any $\calL\in \pic(Z)$ and any $l'\in L'(\mathcal{T})$ it satisfies
\begin{equation}\label{eq:CHERNREST}
c_1(r(\calL))=R_1(c_1(\calL)).
\end{equation}

Similarly assume that $\calv \setminus |Z_1| = \calv_2$ and let us denote the cohomological restriction operator by $R_2 : L'(\mathcal{T}) \to L_2':=L'(\mathcal{T}_2)$
defined as $R_2(E^*_v(\mathcal{T}))=E^*_v(\mathcal{T}_2)$ if $v\in \calv_2$, and
$R_2(E^*_v(\mathcal{T}))=0$ otherwise.

In particular,
we have the following commutative diagram as well:

\begin{equation*}  
\begin{picture}(200,40)(30,0)
\put(50,37){\makebox(0,0)[l]{$
\ \ \eca^{l'}(Z)\ \ \ \ \ \stackrel{c^{l'}(Z)}{\longrightarrow} \ \ \ \pic^{l'}(Z)$}}
\put(50,8){\makebox(0,0)[l]{$
\eca^{R_1(l')}(Z_1)\ \ \stackrel{c^{R_1(l')}(Z_1)}{\longrightarrow} \  \pic^{R_1(l')}(Z_1)$}}
\put(162,22){\makebox(0,0){$\downarrow \, $\tiny{$r$}}}
\put(78,22){\makebox(0,0){$\downarrow \, $\tiny{$\fr$}}}
\end{picture}
\end{equation*}

By the `relative case' we mean that instead of the `total' Abel map
$c^{l'}(Z)$ (with $l'\in -\calS'$ and  $Z\geq E$)
we study its restriction above a fixed fiber of $r$.
That is, we fix some line bundle  $\mfl\in \pic^{R_1(l')}(Z_1)$, and we study the restriction of the map $c^{l'}(Z)$, namely: $(r\circ c^{l'}(Z))^{-1}(\mfl)\to r^{-1}(\mfl)$.

\begin{definition} With the notations above let us denote the subvariety $(r\circ c^{l'}(Z))^{-1}(\mfl)
=(c^{R_1(l')}(Z_1) \circ \fr)^{-1}(\mfl) \subset \eca^{l'}(Z)$ by $\eca^{l', \mfl}(Z)$.
\end{definition}

We start with some properties of the $\fr$.
Before we state them let us mention that for arbitrary $l'$ the map $\fr$ is not necessarily surjective. 
In fact, it can happen that if we set $\mfl:=c^{R_1(l')}(Z_1)(D_1)$, then even $\im \fr \cap (c^{R_1(l')}(Z_1))^{-1}(\mfl)$ is empty.

However, we have  the following lemma:
\begin{proposition}\label{lem:locsubmer} (a) $\fr$ is a local submersion, that is for any
$D\in \eca ^{l'}(Z)$ and $D_1:=\fr(D)$, the tangent map $T_D\fr$ is surjective.

(b) $\fr$ is dominant.

(c) any non--empty fiber of  $\fr$ is smooth of dimension
$(l', Z)-(l',Z_1)=(l',Z_2)$, and it is irreducible.

\end{proposition}
(Since $\fr$ is not proper, we do not expect that $\fr$ is a $C^\infty$ locally trivial fibration.)
\begin{proof}

{\it (a)}  By the product structure of the
local neighbourhoods of nonconnected divisors in $\eca^{l'}(Z)$, we can assume that $D$ is supported in only one point $p$.

\medskip

We distinguish several different cases in the following:

\medskip

\emph{Assume that $p$ is a regular point of an exceptional divisor $E_v$ such that $v \notin |Z_1|$.}

\medskip

In this case the statement is trivial, since $\fr(D)$ is the empty divisor which is the only point of $\eca^0(Z_1)$ and so the map $\fr$ is obviously a submersion at the point $D$.

\bigskip

\emph{Assume that $p$ is a regular point of an exceptional divisor $E_v$ such that $v \in |Z_1|$ or $p$ is the intersection point of two exceptional divisors $E_v, E_w$,
such that $v, w \in |Z_1|$.}

\medskip

In this case the claim follows from the fact from \cite{NNA1}, that if $A \geq B \geq E$ are effective integer cycles and $l' \in -S'$ arbitrary,
then $\eca^{l'}(A)$ is the total space of a locally trivial fibration over $\eca^{l'}(B)$ with fibres isomorphic to an affine vector space.

\bigskip

\emph{Assume that $p$ is the intersection point of two exceptional divisors $E_v, E_w$, such that $v \in |Z_1|$
and $w \notin |Z_1|$.}

\medskip

Let us call these kind of points contact points in the following, so assume that $p$ is a contact point.

Assume that $D$ has local equation $f=x^n+y^m+xyg(x,y)$, the local equation of $Z$ and $Z_1$ are $x^Ny^M$, $x^{N_1}$ respectively,
where we obviously we have $N_1 \leq N$, $D_1$ is represented by the same equation $f$, but now modulo $x^{N_1}$. Via the proof of Theorem \ref{th:smooth} from \cite{NNA1}, $T_D\eca^{l'}(Z)$ can be identified with the vector space $\C\{x,y\}/(x^Ny^M,f)$, while $T_{D_1}\eca^{R_1(l')}(Z_1)$ can be identified with  $\C\{x,y\}/(x^{N_1},f)$.
Then the natural epimorphism $\C\{x,y\}/(x^Ny^M,f)\to \C\{x,y\}/(x^{N_1},f)$ is  $T_D\fr$, which finishes the proof of part {\it (a)}.

\medskip

{\it (b)} Let $\eca^{R_1(l')}(Z_1)_{\not=cp}$ be the open subset of those divisors in $\eca^{R_1(l')}(Z_1)$ whose supports do not contain any contact points.
 This is a Zariski open set of $\eca^{R_1(l')}(Z_1)$, and for any $D_1\in \eca^{R_1(l')}(Z_1)_{\not=cp}$ the fiber
$\fr^{-1}(D_1)$ is clearly non--empty.

\medskip

{\it (c)} The first two statements follow from {\it (a)}, next we prove the irreducibility.

Again, by the product structure, we can assume that $D_1\in \eca^{R_1(l')}(Z_1)$ is supported in only one point $p$. 
If  $D_1\in \eca^{R_1(l')}(Z_1)_{\not=cp}$ then the statement is again trivial as in part a). 

\medskip

Next assume that $p$ is a contact point:

\bigskip

We will use the local charts from the proof of Theorem \ref{th:smooth} in \cite{NNA1}.

Assume that $p=E_u\cap E_v$, $v\in\calv_1$, $u\in\calv_2$, and we fix local coordinates $(x,y)$ on a neighborhood $U$ of the contact point $p$, such that $\{x=0\}=U\cap E_v$, $\{y=0\}=U\cap E_u$.
Assume that $Z$ has local equation $x^N y^M$ and $Z_1$ has got local equation $x^{N_1}$, where we have $1 \leq N_1 \leq N$. 
Assume that $D_1$ has local equation $\widetilde{f}=y^m+x\widetilde{g}(x,y)$ (mod $x^{N_1}$), in particular, $(\widetilde{f}, x)_p=m$ (which is independent of the representative  mod $x^{N_1}$),
and $-m E^*_v$ is the contribution of $E_v$ into the Chern class $l'$ (i.e. $(l',E_v)=m$).

\medskip

We distinguish two cases in the following:

\medskip

\emph{Assume that $\widetilde{g}(x,y)$ has a non--zero monomial of type $x^{o-1}$ with $o<N_1$:}

\medskip

Let $o$ be the smallest one, hence $\widetilde{f}=y^m+xyg(x,y)+\lambda_ox^o+\lambda_{o+1}x^{o+1}+\ldots$ (mod $x^{N_1}$), $\lambda_o \not=0$.
If $(l',E_u)<o$ then we get $\fr^{-1}(D_1)=\emptyset$, on the other hand if $(l',E_u)\geq o$ then set $l'_2:=R_2(l') - oE^*_u$.

Let us denote the cycle $Z' = \max(Z_1, E)$ and consider the map $\eta: \eca^{-mE_v^* - oE_u^*}(Z') \to \eca^{-mE_v^*}(Z_1)$. 

We claim that $\eta^{-1}(D_1)$ is an irreducible 
subvariety of $\eca^{-mE_v^* - oE_u^*}(Z')$ of dimension $o$, let us denote this space by $T_{D_1, o}$.
Since every divisor in $T_{D_1, o}$ has got locally at $p$ intersection multiplicity $o$ with the exceptional divisor $E_u$, we get that every divisor in $T_{D_1, o}$ is supported at the point 
$p$.

Notice that by \cite{NNA1} the divisors at the point $p$ in $\eca^{-mE_v^*}(Z_1)$, which has got intersection multiplicity $m$ with the exceptional divisor $E_v$ can be coordinated as 
$V(y^m + \sum_{1 \leq i \leq Z_v-1, 0 \leq j \leq m-1} a_{i, j} x^i y^j)$, where $V(g)$ denotes the divisor of a function $g$.
The divisors at the point $p$ in $\eca^{-mE_v^* - oE_u^*}(Z')$, which has got locally at $p$ intersection multiplicity $o$ with the exceptional divisor $E_u$ and intersection multiplicity $m$ with the exceptional divisor $E_v$ can be coordinated as $V(y^m + \sum_{1 \leq i \leq Z_v-1, 1 \leq j \leq m-1} a_{i, j} x^i y^j  + \sum_{o \leq l \leq Z_v + o -1}b_l  x^l)$.
It gives that the space $T_{D_1, o} = \eta^{-1}(D_1)$ is isomorphic to an affine space parametrised by the coordinates $b_l , Z_v \leq l \leq Z_v + o -1$.

It means that the inverse image of the divisor $D_1$ at the map $h: \eca^{l'}(Z') \to \eca^{l'}(Z_1)$ is $h^{-1}(D_1) = T_{D_1, o} \times \eca^{l'_2}(E_2)_{\not=cp}$,
so it is smooth, irreducible and of dimension $(l',Z_2)-  (l', Z - Z')$.

Consider also the map $h':  \eca^{l'}(Z) \to \eca^{l'}(Z')$, then we have $\fr = h \circ h'$ and notice that by \cite{NNA1} the map $h':  \eca^{l'}(Z) \to \eca^{l'}(Z')$ is a locally trivial fiber bundle with fiber isomorphic to an affine space of dimension $(l', Z - Z')$.
On the other hand we know that $h^{-1}(D_1)$ is smooth and irreducible, so we indeed get that $\fr^{-1}(D_1) = h'^{-1}(h^{-1}(D_1))$ is also smooth, irreducible and of the right 
dimension $(l',Z_2)$.

\medskip

\emph{Assume that such an $o<N_1$ does not exists, that is, $\widetilde{f}=y^m+xyg(x,y)$ (mod $x^{N_1}$):}

\medskip

 If $(l', E_u)<N_1$ then again $\fr^{-1}(D_1)= \emptyset$ since any extension $D$ contributes with at least $N_1$ into the intersection multiplicity $(l', E_u)$, so assume in the following that $N_1 \leq (l', E_u)$.

For any integer $N_1 \leq o\leq (l', E_u)$  let us denote the subset $(\fr^{-1}(D_1))_o \subset \fr^{-1}(D_1)$ consisting of the divisors, whose restriction at the point p interesects 
the exceptional divisor $E_u$ with multiplicity $o$.

This means that a local equation of the divisor looks like $f_o=y^m+xyg(x,y)+\lambda_ox^o$ ($\lambda_o\not=0$).
We claim that $(\fr^{-1}(D_1))_o \in \overline{(\fr^{-1}(D_1))_{N_1}}$ if $N_1 < o \leq (l', E_u)$. Indeed notice if we have a divisor $D \in (\fr^{-1}(D_1))_o$ with local equation at $p$ $f_o = y^m+xyg(x,y)+ \lambda_{o}x^{o}$, then it can be deformed locally to $y^m+xyg(x,y)+ \lambda_{o}x^{o} +  \lambda_{N_1}x^{N_1}$ and if $\lambda_{N_1}$ is enough small, we get a divisor in $(\fr^{-1}(D_1))_{N_1}$, this indeed proves that $(\fr^{-1}(D_1))_o \in \overline{(\fr^{-1}(D_1))_{N_1}}$.

It means that we have to prove that $(\fr^{-1}(D_1))_{N_1}$ is irreducible.

Consider the irreducible subspace of divisors $T_{D_1, N_1} \subset \eca^{-mE_v^* - N_1 E_u^*}(Z')$ on $Z'$ supported at the point $p$ as in the previous case and let us have the subspace $T_{N_1} = T_{D_1, N_1} \times \eca^{R_2(l') - N_1 E^*_u}(E_2)_{\not=cp} \subset
\eca^{l'}(Z')$, which is smooth and irreducible and notice that $h'^{-1}(T_{N_1}) = (\fr^{-1}(D_1))_{N_1}$.

This indeed proves that $(\fr^{-1}(D_1))_{N_1}$ is irreducible, since $h'$ is a locally trivial fiber bundle with fiber isomorphic to an affine space, this proves part c) completely.

\end{proof}

\begin{corollary}\label{cor:smoothirreddim}
Fix an arbitrary singularity with resolution $\tX$ and with rational homology sphere resolution graph, a Chern class $l'\in -\calS'$, an integer effective cycle $Z\geq E$ and a cycle $Z_1 \leq Z$. Consider a line bundle  $\mfl\in \pic^{R(l')}(Z_1)$.
Assume that  $\eca^{l', \mfl}(Z)$ is nonempty, then it is smooth of dimension $h^1(Z_1,\mfl)  - h^1(\calO_{Z_1})+ (l', Z)$ and irreducible.
\end{corollary}

\begin{proof}
Proposition
 \ref{lem:locsubmer}{\it (a)} together with the fact that $c^{R_1(l')}(Z_1)^{-1}(\mfl)$ is smooth
 shows that $\eca^{l', \mfl}(Z)$ is smooth whenever it is non--empty.
Its dimension is $\dim (c^{R_1(l')}(Z_1)^{-1}(\mfl)) + \dim \eca^{l'}(Z)-\dim \eca^{R_1(l')}(Z_1)$. Then use
(\ref{eq:dimfiber}) and Theorem \ref{th:smooth}{\it (1)}.

Since $\fr$ is dominant and open (cf. \ref{lem:locsubmer}), if $\eca^{l',\mfl}(Z)$ is non--empty,
that is ${\mathfrak{R}}:=\im \fr\cap c^{R_1(l')}(Z_1)^{-1}(\mfl)\not=\emptyset$,   then
${\mathfrak{R}}$ is a non--empty Zariski open set in $c^{R_1(l')}(Z_1)^{-1}(\mfl)$,
hence it is irreducible. 
Since all the fibers of $\fr$ over ${\mathfrak{R}}$ are irreducible (cf.  \ref{lem:locsubmer}{\it (c)}) and $\fr$ is a local submersion, we get indeed that $\eca^{l',\mfl}(Z)$ itself is irreducible. 
\end{proof}

Next we wish to prove a characterisation of the dominance of Abel maps in the relative setup:

\begin{definition}
Fix an arbitrary singularity with resolution $\tX$ and with rational homology sphere resolution graph, a Chern class $l'\in -\calS'$, an integer effective cycle $Z\geq E$, a cycle $Z_1 \leq Z$ and a line bundle $\mfl \in \pic^{R_1(l')}(Z_1)$ as above.
We say that the pair $(l',\mfl ) $ is {\it relative
dominant} on the cycle $Z$, if the closure of $ r^{-1}(\mfl)\cap \im(c^{l'})$ is $r^{-1}(\mfl)$.
\end{definition}

\begin{theorem}\label{th:dominantrel} We have the following two statements about relative dominancy:

(1) If $(l',\mfl)$ is relative dominant on the cycle $Z$, then $ \eca^{l', \mfl}(Z)$ is
nonempty and $h^1(Z,\calL)= h^1(Z_1,\mfl)$ for any
generic line bundle $\calL\in r^{-1}(\mfl)$.

(2) $(l',\mfl)$ is relative dominant on the cycle $Z$ if and only if for all
 $0<l\leq Z$, $l\in L$ one has
$$\chi(-l')- h^1(Z_1, \mfl) < \chi(-l'+l)-
 h^1((Z-l)_1, \mfl(-l)),$$ where we denote $(Z-l)_1 = \min(Z-l, Z_1)$.
\end{theorem}

\begin{proof} {\it (1)} Since $c^{l'}(Z) : \eca^{l', \mfl}(Z) \to  r^{-1}(\mfl)$ is dominant,
$\eca^{l', \mfl}(Z)$ is non--empty.

 Furthermore, for a generic line bundle $\calL$ in $r^{-1}(\mfl)$,
the dimension of the fiber of $ \eca^{l', \mfl}(Z)  \to  r^{-1}(\mfl)$ is
$\dim\eca^{l'\mfl}(Z) - \dim( r^{-1}(\mfl))$, or $(l', Z)+  h^1(Z_1,\mfl) 
 - h^1(\calO_{Z_1}) - ( h^1(\calO_Z) - h^1(\calO_{Z_1})) = (l', Z)+ h^1(Z_1,\mfl) - h^1(\calO_Z)$.

On the other hand, by (\ref{eq:dimfiber}),  this dimension is 
 $(l', Z)+ h^1(Z,\calL) - h^1(\calO_Z)$, hence we get indeed $h^1(Z,\calL)= h^1(Z_1,\mfl)$.

\medskip

{\it  (2)}\  Assume that $(l',\mfl)$ is dominant on the cycle $Z$, but
 $\chi(-l')-  h^1(Z_1, \mfl) \geq \chi(-l'+l)-
  h^1((Z-l)_1, \mfl(-l)) $ $(\dag)$ for some $0<l\leq Z$.

  Choose some generic line bundle $\calL\in r^{-1}(\mfl)\cap \im(c^{l'}(Z))$, we know that $(\dag)$ is equivalent with  $ h^1(Z_1, \mfl) + \chi(Z, \calL) \leq  h^1((Z-l)_1, \mfl(-l)) + \chi(Z-l, \calL (-l))$.

Next, since  $(l',\mfl)$ is dominant, from part  {\it (1)} we get that
$  h^0(Z, \calL)= h^1(Z_1,\mfl ) + \chi(Z, \calL)$.
 On the other hand,  $ h^1((Z-l)_1, \mfl(-l)) \leq
  h^1(Z-l, \calL(-l))$, or equivalently
   $ h^1((Z-l)_1, \mfl(-l)) + \chi(Z-l, \calL (-l)) \leq  h^0(Z-l, \calL(-l))$.

All these combined provide
 $ h^0(Z, \calL) \leq h^0(Z-l, \calL(-l))$ for some $l>0$,
 which yields $H^0(Z,\calL)_{reg}=\emptyset$.
 This contradicts the fact that $\calL\in
 \im(c^{l'}(Z))$.

For the opposite direction, assume that $\chi(-l')- h^1(Z_1, \mfl) < \chi(-l'+l)-
 h^1((Z-l)_1, \mfl(-l))$, or equivalently,
$ h^1((Z-l)_1, \mfl(-l)) + \chi(Z-l, \calL(-l)) <  h^1(Z_1, \mfl) + \chi(Z, \calL)$, for all $0<l\leq Z$.

  This, for $l=Z$, and any line bundle $\calL\in r^{-1}(\mfl)$ implies $\chi(Z,\calL)>-h^1(Z_1,\mfl)$, or,
  $h^0(Z,\calL)>h^1(Z,\calL)-h^1(Z_1,\mfl)$.

  Since $\mfl=\calL|_{Z_1}$, the epimorphism of sheaves
 $\calL\to \calL|_{Z_1}$ induces an epimorphism
 $H^1(Z,\calL)\to H^1(Z_1,\mfl)$, hence $h^0(Z,\calL)>0$.

Assume in the following that $\calL$ is a generic element of $ r^{-1}(\mfl)$.
If $H^0(Z,\calL)_{reg}=\emptyset$, then there exists $E_v$ such that
 $H^0(Z,\calL)=H^0(Z-E_v,\calL(-E_v))$. If $H^0(Z-E_v,\calL(-E_v))_{reg}=\emptyset$ again,
 then we continue the procedure. Finally we obtain
 a cycle $0<l\leq Z$ such that  $H^0(Z-l, \calL(-l))=H^0(Z,\calL)$
 and $H^0(Z-l, \calL(-l))_{reg}\not=\emptyset$, or in other words $\calL(-l)\in \im(c^{l'-l}(Z-l))$.

  Moreover, since $\calL $ is
  generic  in  $r^{-1}(\mfl)$, one gets that
  $\calL(-l)|_{Z-l} \in\pic^{l'-l}(Z-l) $ is generic
  in $r^{-1} (\mfl(-l))$ too,
  where $r$ is again the natural restriction map
  $\pic^{l'-l}(Z-l) \to\pic((Z- l)_1)$.
This shows that the pair $(l'-l,\mfl(-l))$ is relative dominant on the cycle $(Z-l)$,
and by part {\it (1)} we obtain $h^1(Z-l,\calL(-l))=h^1((Z-l)_1,\mfl(-l))$.
All these facts together imply
$ h^0(Z, \calL)=  h^0(Z-l, \calL(-l)) = h^1((Z-l)_1,
\mfl(-l)) + \chi(Z-l, \calL(-l)) <
h^1(Z_1, \mfl) + \chi(Z, \calL)$. This simplifies into $h^1(Z,\calL)<h^1(Z_1,\mfl)$,
which is false, so this contradiction finishes the proof of part 2).

\end{proof}

We are ready to prove the main theorem of the section Theorem \textbf{A}, in fact we prove it first in a more general form for cycles instead of the resolution space $\tX$:

\begin{theorem}\label{th:hegy2rel}
Fix an arbitrary singularity $\tX$ with rational homology sphere resolution graph, a Chern class $l'\in -\calS'$, an integer effective cycle $Z\geq E$, a cycle $Z_1 \leq Z$ and a line bundle $\mfl\in \pic^{R_1(l')}(Z_1)$ as in Theorem \ref{th:dominantrel}. 
Then for any line bundle $\calL\in r^{-1}(\mfl)$ one has
\begin{equation*}\label{eq:genericLrel}
\begin{array}{ll}h^1(Z,\calL)\geq \chi(-l')-
\min_{0\leq l\leq Z,\ l\in L} \{\,\chi(-l'+l) -
h^1((Z-l)_1, \mfl(-l))\, \}, \ \ \mbox{or, equivalently,}\\
h^0(Z,\calL)\geq \max_{0\leq l\leq Z,\, l\in L}
\{\,\chi(Z-l,\calL(-l))+  h^1((Z-l)_1, \mfl(-l))\,\}.\end{array}\end{equation*}
Furthermore, if $\calL$ is generic in $r^{-1}(\mfl)$
then in both inequalities we have equalities.
\end{theorem}

\begin{proof} Again, by Riemann--Roch,
it is enough to verify only the
statement for $h^0$. 

Note that for any $l$ and $\calL$  one has
$h^0(Z,\calL)\geq h^0(Z-l,\calL(-l))= \chi(Z-l,\calL(-l)) +
 h^1(Z-l,\calL(-l)) \geq \chi(Z-l,\calL(-l)) +  h^1((Z-l)_1,\mfl(-l)) $, hence the inequality follows.
We need to show the opposite inequality for $\calL$ generic in $r^{-1}(\mfl)$.
For $h^0(Z,\calL)=0$ it follows  by taking $l=Z$.

Hence, assume $h^0(Z,\calL)\not=0$. Then, as in the proof of  Theorem \ref{th:dominantrel},
there exists a cycle $0\leq l<Z$ such that $h^0(Z,\calL)=h^0(Z-l,\calL(-l))$ and
$H^0(Z-l,\calL(-l))_{reg}\not=\emptyset$. Moreover, $l'-l\in -\calS'_{|Z|}$ (cf. (\ref{eq:empty})),
and by Theorem \ref{th:dominantrel}
$h^1(Z-l,\calL(-l))= h^1((Z-l)_1, \mfl(-l))$ as well.
Hence $h^0(Z,\calL)=\chi(Z-l,\calL(-l)) + h^1((Z-l)_1, \mfl(-l))
\leq \max_{0\leq l\leq Z,\, l\in L}\{\,\chi(Z-l,\calL(-l)) +
h^1((Z-l)_1, \mfl(-l))\,\}$.
\end{proof}

\begin{remark}\label{eleg}
Notice that from the proof above we also proved the following bit stronger statement: if  $\calL$ is generic in $r^{-1}(\mfl)$ and $h^0(Z,\calL)\not=0$, then $h^0(Z,\calL) = \max_{0\leq l\leq Z,\, l\in L, H^0(Z-l, \calL-l)_{reg} \neq \emptyset}\{\,\chi(Z-l,\calL(-l))+  h^1((Z-l)_1,  \mfl(-l))\,\}$. Using the formal neighborhood theorem again one gets a similar sharpening of Theorem \textbf{A} for the resolution space $\tX$ too.
\end{remark}

\begin{proof}[Proof of Theorem \textbf{A}]
It is easy to see that in both statements of Theorem\ref{th:hegy2rel} the minimums and maximums are realised for a bounded region of the cycles $0 \leq l \in L$ independently of the cycle $Z$. On the other hand if the cycle $Z$ is very large (both of its coefficients) then from the formal neighborhood theorem
we get that $h^1(Z, \calL) = h^1(\tX, \calL)$ and $h^1((Z-l)_1, \mfl(-l)) = h^1(\tX_1, \mfl(-l))$ for all such cycles $l$, which proves Theorem \textbf{A}.
\end{proof}

\begin{remark}\label{rem:relative}

Let us look at the special case in the following, when $Z = Z_1 + Z_2$, where $|Z_1| \cap |Z_2| = \emptyset$ and let us denote $|Z_1| = \calv_1$
and $|Z_1| = \calv_2$, and the corresponding subgraphs by $\mathcal{T}_1$ and $\mathcal{T}_2$.

We have the cohomological restriction operators $R_i : L'(\mathcal{T}) \to L_i':=L'(\mathcal{T}_i)$ for $i =1,2$, defined as $R_i(E^*_v(\mathcal{T}))=E^*_v(\mathcal{T}_i)$ if $v\in \calv_i$, and $R_i(E^*_v(\mathcal{T}))=0$ otherwise.

(1) The inequalities from Theorem \ref{th:hegy2rel} provide (in principle)
better bounds for elements $\calL\in r^{-1}(\mfl)$ then the trivial bound $h^1(Z, \calL) \geq \chi(-l') - \min_{0 \leq l \leq Z} \chi(-l' + l)$ valid for arbitrary
line bundles $\calL\in\pic^{l'}(Z)$ by semicontinuity.

However, in some cases they coincide. E.g., if $\mathcal{T}_1$ is rational, then $\pic^{l'}(Z_1)=
\{\mfl\}=\{\calO_{Z_1}(l')\}$, hence  $r^{-1}(\mfl)=\pic^{l'}(Z)$. Since both
inequalities become equalities  for generic elements, one has
$$\min_{0\leq l\leq Z}\, \chi(-l'+l) = \min_{0\leq l\leq Z}\,
\{ \, \chi(-l'+l)-h^1(\calO_{(Z-l)_1}(l'-l))\,\},$$
a fact, which is not at all evident without Theorems \ref{th:dominantrel} and \ref{th:hegy2rel}.

(2) If we take in the expression in $\min$ on  the right hand side of (\ref{eq:genericLrel})
$l=0$ we obtain $h^1(Z,\calL)\geq h^1(Z_1,\mfl)$ (a fact, which follows from the
epimorphism $\calL\to \calL|_{Z_1}$ too). However, the actual bound
of (\ref{eq:genericLrel}) in general is strict larger. Indeed,
already a  better estimate
 is given by cycles $l$ with $l_1=0$. For them   we get
\begin{equation}\label{eq:legy}
h^1(Z,\calL)  \geq  \max_{0\leq l_2\leq Z_2}\{
h^1(Z_1,\mfl(-l_2))- \chi(l_2)-(l',l_2)\}.
\end{equation}
Here, usually (if $Z_2\gg 0$) both terms
$h^1(Z_1,\mfl(-l_2))$ and $\chi(l_2)+(l',l_2)$ increases (tend to $\infty$)
when $l_2$ increases, and  their difference provides the bound in (\ref{eq:legy}).

In fact, one can split this expression into a bound which shows
some basic independent $\mathcal{T}_1$-$\mathcal{T}_2$--contributions
(however, in this way  it becomes less sharp). Indeed,
since the inclusion $\mfl(-l_2)\to \mfl$
has finite quotient, $h^1(Z_1,\mfl(-l_2) \geq h^1(Z_1,\mfl)$. This shows
\begin{equation}\label{eq:legy2}
\begin{split}
h^1(Z,\calL) & \geq
h^1(Z_1,\mfl)-\min_{0\leq l_2\leq Z_2}\{ \chi(l_2)+(l',l_2)\}  \\
  &=h^1(Z_1,\mfl)+\min\{ h^1(Z_2,\calL_2)\,:\, \calL_2\in \pic^{R_2(l')}(Z_2)\}.
\end{split}
\end{equation}
E.g., if $l'=0$  then
$h^1(Z,\calL)\geq  h^1(Z_1,\mfl)-\min\chi(\mathcal{T}_2)$
(which can be much  larger than $h^1(Z_1,\mfl)$).
However, this is much
smaller than the actual bound of (\ref{eq:genericLrel}). Assume e.g.
that both $\mathcal{T}_1$ and $\mathcal{T}_2$ are rational and $l'=0$, then (\ref{eq:legy2})
provides the trivial bound, but in   (\ref{eq:genericLrel}),
$-\min\chi(\mathcal{T})$ can be arbitrary large.
This somehow suggests that the expression from the right hand side of
(\ref{eq:genericLrel}) is rather complex, and it motivates its `non--easy' form as well
(saying that we cannot take a naive splitting into $\mathcal{T}_1$ and $\mathcal{T}_2$).
\end{remark}

\section{Working definition of relatively generic normal surface singularities} 

In this section we define relatively generic analytic structures on normal surface singularities.

\medskip

\emph{Let us explain first in philosophical terms what we will do:}

\medskip

Consider a rational homology sphere resolution graph $\mathcal{T}$ with vertex set $\calv$, and we consider a partition $\calv = \calv_1 \cup  \calv_2$ of the set of vertices $\calv=\calv(\mathcal{T})$.  Let us denote the possibly nonconnected subgraph of $\mathcal{T}$ supported on the vertex set $\calv_i$ by $\mathcal{T}_i$ ($i=1,2$). For any integer cycle $Z\in L=L(\mathcal{T})$ we write $Z=Z_1+Z_2$, where $Z_i\in L(\mathcal{T}_i)$ is supported in $\mathcal{T}_i$.
Fix a normal surface singularity $\tX_1$ with resolution graph $\mathcal{T}_1$ and for each vertex $v_2 \in \calv_2$ which has got a neighbour $v_1$ in $\calv_1$ fix a cut $D_{v_2}$ on $\tX_1$, which is a smooth irreducible divisor, which intersects the exceptional divisor $E_{v_1}$ transversally in one point.
We will glue the tubular neighborhood of the exceptional divisor $E_{v_2}$ in a way, such that $E_{v_2} \cap \tX_1 = D_{v_2}$.
If we plumb the tubular neighbourhoods of the exceptional divisors $E_{v_2}, v_2 \in \calv_2$  with the above conditions "generically" to the fixed resolution $\tX_1$, then we get a singularity $\tX$ with resolution graph $\mathcal{T}$ and we say that $\tX$ is a relatively generic singularity corresponding to the analytical structure $\tX_1$ and the cuts $D_{v_2}$.

\medskip

If for some vertex $v_2 \in \calv_2$, which has a neighbour in $\calv_1$ we do not say what is the fixed cut $D_{v_2}$, then it should be understood in the way that we glue the exceptional divisor $E_{v_2}$ along a generic cut. Notice that for a vertex $v_2 \in \calv_2$ the space of the possible cuts $D_{v_2}$ is the inverse limit of the spaces 
$\eca^{-n E_{v_2}^*}(Z)$ with the natural maps $\eca^{-n E_{v_2}^*}(Z_1) \to \eca^{-n E_{v_2}^*}(Z_2), Z_1 \geq Z_2$ between them, where $Z$ runs through effective integer cycles.
We say that a cut $D_{v_2}$ is generic, if for a large cycle $Z \gg 0$ the projection of $D_{v_2}$ into $\eca^{-n E_{v_2}^*}(Z)$ lies in a Zariski open subset (suitable for the corresponding problem).

\medskip

Next, we define more precisely what we mean by a relatively generic singularity using the definitions of Laufer from section 4.
Consider the set of vertices in $\calv \setminus \calv_1$, which have a neighbour in the subgraph $\calv_1$ and let us denote it by $N(\calv_1)$.
Fix a normal surface singularity $\tX_1$ with resolution graph $\mathcal{T}_1$ and for each vertex $v_2 \in N(\calv_1)$ we fix a cut $D_{v_2}$ on $\tX_1$.
Consider a very large cycle $Y$ on $\mathcal{T}$ with a fixed analytic type, such that the analytic structure of $Y$ determines the analytic structure on the resolution,
and consider its subcycle $Y_1$. Assume also that $Y_1$ determines the analytic type $\tX_1$ on $\mathcal{T}_1$ and the analytic type of $Y$ is compatible with the cuts $D_{v_2}$.
Let us denote the effective integer cycle $Y' = Y_1 + \sum_{v \in N(\calv_1)} E_{v}$ and with the analytic type we get by glueing the reduced exceptional divisors $E_{v_2}, v_2 \in N(\calv_1)$ to $Y_1$ along the cuts $D_{v_2}$.

Assume that we have a complete deformation $\omegl:\caly \to Q$ of $Y$ fixing $Y'$ and choose a generic point $q \in Q$ in the finite dimensional parameter space, then we have the
pair  $Y_q , Y'$. The notion of the chosen generic point $q \in Q$ depends on the suitable discriminants we want to avoid in different geometric problems, in this manuscript it will be a union of loci where the $h^1$ of some special line bundles jump, similarly as in \cite{NNA2}.

\begin{definition}
The cycle $Y_q$ determines a unique analytic type $\tX_q$ where we have $E_{v_2} \cap \tX_1 = D_{v_2}$ for every $v_2 \in N(\calv_1)$. We call $\tX_q$ a relatively generic analytic structure corresponding to $\tX_1$ and the cuts $D_{v_2}$.
\end{definition}

\begin{remark}
From the universal property of complete deformations one can see that this definition does not depend on the choosen local complete deformation, just on the original fixed analytic type.
\end{remark}

\begin{remark}
If for a subset of vertices $v_2 \in J, J \subset N(\calv_1)$ we do not a fix a cut $D_{v_2}$, then we can choose $D_{v_2}$ as a generic cut described earlier.
Equivalently one can consider the cycle $Y' = Y_1 + \sum_{v \in N(\calv_1) \setminus J} E_{v}$ and choose $Y_q$ (and thus $\tX_q$) as a generic element from the parameter space
of a complete deformation of the pair $(Y, Y')$.
\end{remark}

Fix some notations in the setup described above:

\medskip

We call the intersection of an exceptional divisor from
$ \calv_1 $ with an exceptional divisor from  $ \calv_2 $ a
{\it contact point}. Furthermore as usual, parallel to the restriction
$r_i : \pic(Z)\to \pic(Z_i)$ one also has the (cohomological) restriction operator
  $R_i : L'(\mathcal{T}) \to L_i':=L'(\mathcal{T}_i)$
(defined as $R_i(E^*_v(\mathcal{T}))=E^*_v(\mathcal{T}_i)$ if $v\in \calv_i$, and
$R_i(E^*_v(\mathcal{T}))=0$ otherwise).
For any $l'\in L'(\mathcal{T})$ and any $\calL\in \pic^{l'}(Z)$ it satisfies $c_1(r_i(\calL))=R_i(c_1(\calL))$.
In the following for the sake of simplicity we will denote $r = r_1$ and $R = R_1$.

\medskip

We will use the following facts about the definition of relatively generic analytic structures in the proofs of our main theorems:

\medskip

Consider three rational homology sphere resolution graphs $\mathcal{T}_1 \subset \mathcal{T} \subset \mathcal{T}'$. Denote the set of vertices $v_2 \in \calv' \setminus \calv_1$ which have a neighbour in $\calv_1$ by $N'(\calv_1)$ and the set of vertices $v_2 \in \calv \setminus \calv_1$ which have a neighbour in $\calv_1$ by $N(\calv_1)$.
Consider a singularity $\tX_1$ with resolution graph $\mathcal{T}_1$ and cuts $D_{v_2}$ for the vertices $v_2 \in N'(\calv_1)$.

\begin{lemma}\label{subgraphrel}
Assume that $\tX'$ is a relatively generic singularity corresponding to $\mathcal{T}'$ and $\tX_1$ and the cuts $D_{v_2}, v_2 \in N'(\calv_1)$.
Then the subsingularity $\tX$ of $\tX'$ with resolution graph $\mathcal{T}$ is a relatively generic singularity corresponding to $\mathcal{T}$, $\tX_1$ and the cuts $D_{v_2}, v_2 \in N(\calv_1)$.
\end{lemma}
\begin{proof}
Consider a fixed singularity with resolution graph $\mathcal{T}'$ which is compatible with $\tX_1$ and the cuts $D_{v_2}$. Consider a large cycle on it $Y$, and its subcycle $Y' = Y_1 + \sum_{v \in N(\calv_1)} E_{v}$. Denote also the subcycle of $Y$ supported on the subgraph $\mathcal{T}$ by $Z$ and its subcycle $Z' = Z_1 + \sum_{v \in N(\calv_1) \cap \calv} E_{v}$.

Assume that we have a complete deformation $\omegl:\caly \to Q$ of $Y$ fixing $Y'$ and choose a generic point $q \in Q$ in the finite dimensional parameter space, then we have the
pair  $Y_q , Y'$. Consider also the subcycle $Z_q \leq Y_q$.
Notice that from the fact that $\omegl$ is complete we get that the Kodaira-Spencer map $T_0Q \to H^1(Y,\theta_{Y,Y'})$ is surjective. Since the map $H^1(Y,\theta_{Y,Y'}) \to H^1(Z,\theta_{Z,Z'})$ is also surjective we get that the composite map $T_0Q \to H^1(Z,\theta_{Z,Z'})$ is surjective.
It means that the deformation $\omegl$, viewed as a deformation of the pair $(Z, Z')$ is also complete. This indeed proves that $Z_q$ and the analytic type on the subgraph $\mathcal{T}$, which is determined by $Z_q$ are also relatively generic.
\end{proof}

In the following lemma we consider two rational homology sphere resolution graphs $\mathcal{T}_1 \subset \mathcal{T}$, a singularity $\tX_1$ with resolution graph $\mathcal{T}_1$ and
cuts $D_{v_2}$ for the vertices $v_2 \in \calv \setminus \calv_1$ which have a neighbour in $\calv_1$.

\begin{lemma}\label{blowup}
Assume that $\tX$ is a relatively generic singularity corresponding to $\mathcal{T}$, $\tX_1$ and the cuts $D_{v_2}$.

1) Blow up a point $p \in E_v$, where $v \in \calv_1$ and $p$ is not a contact point, and let the new singularities be $\tX_{new}, \tX_{1, new}$ with resolution graphs $\mathcal{T}_{new}, \mathcal{T}_{1, new}$. Then $\tX_{new}$ is a relatively generic singularity corresponding to $\mathcal{T}_{new} , \tX_{1, new}$ and the cuts $D_{v_2}$.

\medskip

2) Blow up a contact point $p = E_{v_1} \cap E_{n(v_1)}$ where $v_1 \in \calv_1, n(v_1) \in \calv_2$, and let the new singularities be $\tX_{new}, \tX_{1, new}$ with resolution graphs $\mathcal{T}_{new}, \mathcal{T}_{1, new}$. Then  $\tX_{new}$ is a relatively generic singularity corresponding to $\mathcal{T}_{new} , \tX_{1, new}$ and the cuts $D_{v_2}, v_2 \neq n(v_1)$ and the strict transfrom of $D_{n(v_1)}$.

\medskip
3) Consider a vertex $v \in \calv_2$ and a generic point $p \in E_v$ and blow up $\tX$ in $p$. Let the new singularity be $\tX_{new}$ with resolution graph $\mathcal{T}_{new}$, then $\tX_{new}$ is a relatively generic singularity corresponding to $\mathcal{T}_{new} , \tX_{1}$ and the cuts $D_{v_2}$.
\end{lemma}
\begin{proof}
The proof of pat 1) and part 2) is straightforward.

For part 3) consider a singularity $\tX_{new}$ with resolution graph $\mathcal{T}_{new}$ which is compatible with the subsingularity $\tX_1$ and the cuts $D_{v_2}$.
Consider also a large cycle $Y_{new}$ on it of the form $\pi^*(Y)$, where $\pi$ is the blow down map.
Consider also the subcycle of $Y_{new}$, $Y' = Y_1 + \sum_{v \in N(\calv_1)} E_{v}$, and a complete deformation $\omegl:\caly_{new} \to Q$ of $Y_{new}$ fixing $Y'$.
For a generic point $q \in Q$, the cycle $Y_{new, q}$ determines a relatively generic analytic type $\tX_{new, q}$.
By blowing down we get also a deformation $\omegl:\caly \to Q$ of the pair $(Y, Y')$ and it is also complete since the natural map $H^1(Y_{new},\theta_{Y_{new},Y'}) \to H^1(Y,\theta_{Y,Y'})$ is surjective (via dimension reason). 
This means that for a generic point $q \in Q$, the cycle $Y_{q}$ determines a relatively generic analytic type $\tX_{q}$, part 3) follows from this immediately.
\end{proof}

\section{Cohomological invariants of relatively generic singularities}

\subsection{Cohomology of natural line bundles}

In this subsection we wish to compute the geometric genus or the cohomology of natural line bundles on a relatively generic singularity.

Let us fix some notations first:

\medskip

Assume that we have two rational homology sphere resolution graphs $\mathcal{T}_1 \subset \mathcal{T} $ with vertex sets $\calv_1 \subset \calv$, where $\calv = \calv_1 \cup \calv_2$  and a fixed singularity $\tX_1$ for the subgraph $\mathcal{T}_1$, and cuts $D_{v_2}$ (for vertices $v_2$ which have a neighbour in $\calv_1$).
Consider an effective cycle $Z$ and write $Z = Z_1 + Z_2$, where $|Z_1| \subset \calv_1$ and $|Z_2| \subset \calv_2$.

We prove the following theorem with this setup:

\begin{theorem}\label{relgen}
Assume that $\tX$ has a relatively generic analytic stucture on $\mathcal{T}$ corresponding to $\tX_1$ and the cuts $D_{v_2}$.

\medskip

1) Consider a natural line bundle $\calL= \calO_{\tX}(l')$ on $\tX$, such that $ l' = - \sum_{v \in \calv} a_v E_v$, with $a_v > 0, v \in \calv_2 \cap |Z|$, and denote $c_1 (\calL | Z) = l'_ m \in L'_{|Z|}$, denote $\mfl = \calL | Z_1$.

We have $H^0(Z,\calL)_{reg} \not=\emptyset$ if and only if $(l',\mfl)$ is relative dominant on the cycle $Z$ or equivalently:

\begin{equation*}
\chi(-l')- h^1(Z_1, \mfl) < \chi(-l'+l)-  h^1((Z-l)_1, \mfl(-l)),
\end{equation*}
for all $0 < l \leq Z$.

\medskip

2)  We have the equality $ h^1(Z, \calL) =  h^1(Z, \calL_{gen})$, where $\calL_{gen}$ is a generic line bundle in $ r^{-1}(\mfl) \subset \pic^{l'_m}(Z)$, or equivalently:

\begin{equation*}
h^1(Z, \calL)= \chi(-l') - \min_{0 \leq l \leq Z}(\chi(-l'+l)-  h^1((Z-l)_1, \mfl(-l))).
\end{equation*}

\end{theorem}

\begin{proof}

We will prove statement 1) and 2) simultaneously by induction on the number $h^1(\calO_Z) - h^1(\calO_{Z_1})$. 

\medskip
Assume first that $h^1(\calO_Z) - h^1(\calO_{Z_1})= 0$.  In this case every line bundle $\calL$ on $Z$ is relatively generic with respect to $\mfl  = \calL | Z_1$ since $\dim(  r^ {-1}(\mfl  ) ) = 0$, so the theorem follows from the results on relatively generic line bundles.

\medskip

Assume in the following that $h^1(\calO_Z) - h^1(\calO_{Z_1})> 0$, this means that there is a vertex $u \in \calv_2 \cap |Z|$, such that $e_{Z}(u) > 0$.
Indeed it follows from $e_Z(\calv_2) = h^1(\calO_Z) - h^1(\calO_{Z_1})> 0$

Let us blow up the exceptional divisor $E_u$ sequentially along generic points, let the new vertices be $u_0 = u, u_1, \cdots, u_t$.
Let $t$ be the maximal integer such that $e_{Z}(u_t) > 0$, but if we blow up $E_{u_t}$ in an arbitrary point and the new exceptional divisor is $E_{u_{t+1}}$, then $e_{Z}( u_{t+1}) = 0$.
This means that all differential forms in $ \frac{H^0(\calO_{\tX}( K+Z)) }{H^0(\calO_{\tX}( K))}$ has a pole on the exceptional divisor $E_{u_t}$ of order at most $1$, but there is at least one differential form, which has got a pole on $E_{u_t}$ of order $1$.

Denote the new resoultion graph by $\mathcal{T}_t$ and the blown up resolution by $\tX_t$ and blow up map by $\pi: \tX_t \to \tX$.
Notice that since we blew up $E_u$ sequentially in generic points, the singularity $\tX_t$ is relatively generic with resolution graph $\mathcal{T}_t$ corresponding to the subsingularity
$\tX_1$ and the cuts $D_{v_2}$ by Lemma\ref{blowup}.

Denote the resolution graph supported by the vertices $\calv \setminus u \cup u_0, \cdots \cup u_{t-1} = \calv_s$ by $\mathcal{T}_s$ , and the singularity corresponding to it by $\tX_s$.
Nottice that if $t = 0$, then $t-1 = -1$ and $u_0 = u$, which means that the notation $\calv \setminus u \cup u_0, \cdots \cup u_{t-1}$ should be understood in the way $\calv \setminus u$, in this case $\mathcal{T}_s$ can be even nonconnected.

We know that $\tX_s$ is a relatively generic singularity with resolution graph $\mathcal{T}_s$ corresponding to $\tX_1$ and the cuts $D_{v_2}$.
Similarly $\tX_t$ is a relatively generic singularity with resolution graph $\mathcal{T}_t$ corresponding to $\tX_s$ and cuts described below:

If $t > 0$ then we should glue the tubular neighborhood of $E_{u_t}$ along a generic cut.
If $t = 0$ and it has got a neighbour $v_1 \in \calv_1$, then we shoud glue there along the fixed cut $D_{v_2}$ and to other components of  $\tX_s$ along generic cuts.
Denote $Z_t = \pi^*(Z)$ and let $Z_s$ be the restriction of the cycle $Z_t$ to the subsingularity $\tX_s$.
We know that $h^1(\calO_{Z_s}) - h^1(\calO_{Z_1}) <  h^1(\calO_Z) - h^1(\calO_{Z_1})$.
This means by the induction hypothesis that if $A \leq Z_s$ is any cycle on $\tX_s$ and $\calL_s$ is any natural line bundle on $\tX_t$ satisfiying the conditions of the theorem, then $h^1(A, \calL_s) = h^1(A, \calL_{gen})$, where $\calL_{gen}$ is a generic line bundle in $r^{-1}( \calL_s | A_1)$, with the notation $A_1 = \min(A, Z_1)$.

\medskip

Consider in the following a natural line bundle $\calL= \calO_{\tX}(l')$ which satisfies the conditions of our theorem.
If we write $ l' = - \sum_{v \in \calv} a_v E_v$, then $a_v > 0$ if $ v \in \calv_2 \cap |Z|$, and denote $c_1(\calL |Z ) = l'_ m= - \sum_{v \in |Z|} b_v E_v$.

\bigskip

\textbf{Proof of part 1) in case the pair $(l'_m,\mfl)$ is not relative dominant on the cycle $Z$:}

Assume first that $(l'_m,\mfl)$ is not relative dominant on the cycle $Z$, in this case we want to prove $H^0(Z,\calL)_{reg}=\emptyset$.

Consider the line bundle $\pi^*(\calL)$ on the blown up singularity $\tX_t$ with Chern class $\pi^*(l')$.

For the sake of simpleness we will denote the restriction maps $\pic^{l'_m}(Z) \to \pic^{R(l'_m)}(Z_1)$, $\pic^{\pi^*(l'_m)}(Z_t) \to \pic^{R(l'_m)}(Z_1)$ and $\pic^{R_s(\pi^*(l'_m))}(Z_s) \to \pic^{R(l'_m)}(Z_1)$ with the same notation $r$.
On the other hand we denote the cohmological restriction operator $L'_t \to L'_s$ by $R_s$ and the restriction map $\pic^{\pi^*(l'_m)}(Z_t) \to \pic^{R_s(\pi^*(l'_m))}(Z_s)$
by $r_s$.

From the fact that $(l'_m,\mfl)$ is not relative dominant on the cycle $Z$ we get that $(\pi^*(l'_m),\mfl)$ is not relative dominant on $Z_t$.
Indeed the isometry between $ \pic^{l'_m}(Z)$ and $\pic^{\pi^*(l'_m)}(Z_t)$ brings $r^{-1}(\mfl) \subset \pic^{l'_m}(Z)$ to $r^{-1}(\mfl)\subset \pic^{\pi^*(l'_m)}(Z_t)$ and  $c^{l'_m}( \eca^{l'_m, \mfl}(Z))$ to $c^{\pi^*(l'_m)}( \eca^{\pi^*(l'_m), \mfl}(Z_t))$.

We are enough to prove $H^0(Z_t, \pi^*(\calL))_{reg}=\emptyset$. Indeed if there were a section $s \in H^0(Z,\calL)_{reg}$, then we would have $H^0(Z_t, \pi^*(\calL))_{reg} \neq \emptyset$, because we blew up $E_u$ sequentially in generic points.

\medskip

Let us denote $\mfl_s = r_s(\pi^*(\calL))$, by induction if $(R_s( \pi^*( l'_m)), \mfl)$ is not relative dominant on $Z_s$, then $H^0(Z_s, \mfl_s)_{reg}=\emptyset$, from which $H^0(Z_t,\pi^*(\calL))_{reg}=\emptyset$.

\medskip

Assume in the following that $(R_s( \pi^*( l'_m)), \mfl)$ is relative dominant on the cycle $Z_s$, we prove first the folowing lemma:

\begin{lemma}
The pair  $(\pi^*(l'_m), \mfl_s)$ is not relative dominant on the cycle $Z_t$.
\end{lemma}

\begin{proof}

Assume to the contrary that $(\pi^*(l'_m), \mfl_s)$ is relative dominant on $Z_t$, then it means:

\begin{equation*}
\chi(-\pi^*(l'_m))- h^1(Z_s, \mfl_s) < \chi(-\pi^*(l'_m)+l)-  h^1((Z_t -l)_s, \mfl_s(-l)),
\end{equation*}
for all $0 < l \leq Z_t$.

Consider a generic line bundle $\calL_{gen}$ in $r^{-1}(\mfl) \in \pic^{R_s( \pi^*( l'_m))}(Z_s)$, by the induction hypothesis we know that for all $0 \leq l \leq Z_t$ we have $h^1((Z_t-l)_s, \mfl_s(-l)) = h^1((Z_t-l)_s, \calL_{gen}(-l))$, which means:

\begin{equation*}
\chi(-\pi^*(l'_m)) - h^1(Z_s, \calL_{gen}) < \chi(-\pi^*(l'_m)+l)-  h^1((Z_t-l)_s, \calL_{gen}(-l)),
\end{equation*}
for all $0 < l \leq Z_t$.
 
It means that for a generic line bundle $\calL_{gen} \in  \pic^{R_s( \pi^*( l'_m))}(Z_s)$ we know that $(\pi^*(l'_m), \calL_{gen})$ is relatively dominant on the cycle $Z_t$. However we also know that $(R_s(\pi^*(l'_m)), \mfl)$ is relative dominant on the cycle $Z_s$, so it follows that  $(\pi^*(l'_m),\mfl)$ is also relative dominant on the cycle $Z_t$, which is a contradiction.

\end{proof}

We will prove from this lemma in the following that $H^0(Z_t, \pi^*(\calL))_{reg}=\emptyset$, assume to the contrary that $H^0(Z_t, \pi^*(\calL))_{reg}\neq \emptyset$.

We will deform the analytic type on $\tX_t$ in the following.
Assume that $N$ is a large number, such that $\dim(\im(c^{-N E_{u_t}^*}(Z_t))) = e_{Z_t}(u_t)$, and blow up the exceptional divisor $E_{u_t}$ in N different generic points, denote the new exceptional divisors by $E_{v_1}, \cdots, E_{v_N}$. 
Let us denote the blown up singularity by $\tX_b$ and the pullback of the cycle $Z_t$ by $Z_b$, then we know that $e_{Z_b}(v_j) = 0$ for all $0 \leq j \leq N$, since the pole of a differential form decreases by at least $1$ at a blow up, and every differental form in $\frac{H^0(\calO_{\tX}(K + Z))}{H^0(\calO_{\tX}(K ))}$ has got a pole along the exceptional divisor $E_{v_t}$ of order at most $1$.

Notice that since we blowed up $\tX_t$ in generic points we know that $H^0(Z_b,\pi^*(\calL))_{reg}\neq \emptyset$.

Denote the subsingularity of this blown up singularity supported by the vertices $\calv \cup u_1, \cdots u_t$ by $\tX_u$. We have $h^1(\calO_{Z_u}) = h^1(\calO_{Z_b}) - \sum_{1 \leq i \leq N} e_{Z_b}(v_i) = h^1(\calO_{Z_b})$, it means in particular that $(R_u(\pi^*(l'_m)), \mfl_s)$ is not relative dominant on the cycle $Z_u$.
We get that for a generic element in $\calL_{gen} \in r_s^{-1}(\mfl_s) \subset \pic^{R_u(\pi^*(l'_m))}(Z_u)$ one has $H^0(Z_u, \calL_{gen})_{reg} = \emptyset$.

In the following fix the singularity $\tX_u$ and deform the analytic type of $\tX_b$ by moving the gluing with the tubular neighborhoods of the exceptional divisors $E_{v_1} , \cdots, E_{v_N}$ in a way that we move the intersection points $E_{v_1} \cap E_{u_t}$ such that the analytic type remains relatively generic.

We show first the following easy lemma:

\begin{lemma}\label{coveropen}
If we move the intersection points  $E_{v_1} \cap E_{u_t}$, then the possible values of the line bundles $r_u(\pi^*(\calL))$ cover an open set in $r_s^{-1}(\mfl_s) \subset \pic^{R_u(\pi^*(l'_m))}(Z_u)$.
\end{lemma}
\begin{proof}
We know that $\pi^*(l'_m)$ has got the same positive coefficents on the vertices $v_1, \cdots v_N$, say $h$, and $\dim(\im(c^{-N E_{u_t}^*}(Z_u))) = e_{Z_u}(u_t)$.

Write the line bundle $r_u(\pi^*(\calL))$ in the form:
$$r_u(\pi^*(\calL)) = \calO_{Z_u}((\pi^*(l'))_u) \otimes \calO_{Z_u}(h \cdot \sum_{1 \leq j \leq N} E_{v_j}).$$

Notice that if we deform the intersection points $E_{v_j} \cap E_{u_t}$, then the line bundle $ \calO_{Z_u}((\pi^*(l'))_u)$ stays the same and the line bundle $\calO_{Z_u}(h \cdot \sum_{1 \leq j \leq N} E_{v_j})$ covers an open set in $h \cdot \im(c^{-N E_{u_t}^*}(Z_u))$ and we know that $\dim( h \cdot \im(c^{-N E_{u_t}^*}(Z_u))) = \dim(\im(c^{-N E_{u_t}^*}(Z_t))) = e_{Z_t}(u_t)$.
It proves indeed that the line bundles $r_u(\pi^*(\calL))$ cover an open set in $r_s^{-1}(\mfl_s) \subset \pic^{R_u(\pi^*(l'_m))}(Z_u)$.

\end{proof}

From lemma\ref{coveropen} we get that after a minor deformation of the analytic type of $\tX_b$, $r_u(\pi^*(\calL))$ is a generic line bundle in $r_s^{-1}(\mfl_s) \subset \pic^{R_u(\pi^*(l'_m))}(Z_u)$,which means $H^0(Z_u, r_u(\pi^*(\calL)))_{reg} = \emptyset$, and so $H^0(Z_b,\pi^*(\calL))_{reg}=\emptyset$.
This is a contradiction since the original analytic type $\tX_b$ was already relatively generic and $H^0(Z_b,\pi^*(\calL))_{reg}\neq \emptyset$ was true before the deformation, so
it must hold also after a minor deformation of the analytic type $\tX_b$.

This contradiction proves part 1) in case the pair $(l'_m,\mfl)$ is not relative dominant on the cycle $Z$.

\bigskip

\textbf{Proof of part 1) in case the pair $(l'_m,\mfl)$ is relative dominant on the cycle $Z$:}

Assume that $(l'_m,\mfl)$ is relative dominant on the cycle $Z$, in this case we want to prove $H^0(Z,\calL)_{reg} \neq \emptyset$.

From the fact that $(l'_m,\mfl)$ is relative dominant on the cycle $Z$ it follows that for a generic line bundle $\calL_{gen} \in r^{-1}(\mfl) \subset \pic^{R_s(l'_m)}(Z_s)$, $(\pi^*(l'_m), \calL_{gen})$ is relatively dominant on the cycle $Z_t$, which means by Theorem \ref{th:dominantrel} that:

\begin{equation*}
\chi(-\pi^*(l'_m))- h^1(Z_s, \calL_{gen}) < \chi(-\pi^*(l'_m)+l)-  h^1((Z_t -l)_s, \calL_{gen}(-l)),
\end{equation*}
for all $0 < l \leq Z_t$.

We know that $h^1(\calO_{Z_s}) - h^1(\calO_{Z_1}) <h^1(\calO_{Z}) - h^1(\calO_{Z_1}) $. It means by the induction hypothesis that for all $0 \leq l \leq Z_t$ we have $h^1((Z_t-l)_s, \mfl_s(-l)) = h^1((Z_t-l)_s, \calL_{gen}(-l))$, from which we get:

\begin{equation*}
\chi(-\pi^*(l'_m))- h^1(Z_s, \mfl_s) < \chi(-\pi^*(l'_m)+l)-  h^1((Z_t-l)_s, \mfl_s(-l)),
\end{equation*}
for all $0 < l \leq Z_t$.

It means that $(\pi^*(l'_m),  \mfl_s)$ is relative dominant on $Z_t$, so for a generic line bundle $ \calL_{gen} \in r_s^{-1}(\mfl_s) \subset \pic^{\pi^*(l'_m)}(Z_t)$ we have $H^0(Z_t, \calL_{gen})_{reg} \neq \emptyset$.

\medskip

In the following denote $n = e_{Z_t}(u_t)$, and fix coordinates $w_1, \cdots, w_n$ on the linear subspace $V_{Z_t}(u_t) \subset H^1(\calO_{Z_t})$.
Let us shift the Abel map $\eca^{-E_{u_t}^*}(Z_t)  \to \pic^{-E_{u_t}^*}(Z_t)$ by a constant vector to have a target in the linear subspace $V_{Z_t}(u_t)$ instead of an affine subspace.
Since every differential form in  $ \frac{H^0(\calO_{\tX}( K+ Z)) }{H^0(\calO_{\tX}(K))}$ has got a pole along the exceptional divisor $E_{u_t}$ of order at most $1$, the shifted Abel map $\eca^{-E_{u_t}^*}(Z_t) \to V_{u_t}(Z_t)$ reduces to a map $f: \bP^1 - \delta_{u_{t}}\cdot points \to V_{u_t}(Z)$, since its value on a divisor $D' \in \eca^{-E_{u_t}^*}(Z_t)$ depends just on the intersection point $D' \cap E_{u_t}$.

Notice that  $\dim(\im(c^{-(n+1) E_{u_t}^*})(Z_t)) = e_{Z_t}(u_t) = n$.
Indeed we know that the clousure of $\im(-c^{E_{u_t}^*})(Z_t)$ is an affine algebraic curve, such that the dimension of its affine clousure is $n$, and since $\im(c^{-(n+1) E_{u_t}^*})(Z_t)$ is its $(n+1)$-fold Minkowski sum with itself our claim follows.

We know that the clousure of the image of the map $f$ is an affine algebraic curve and the linearisation of its affine hull is $V_{Z_t}(u_t)$. It means that if we choose $n+1$ generic points $p_1, \cdots, p_{n+1} \in \bP^1 - \delta_{u_{t}} \cdot points$, then $ f'(p_1), ..., f'(p_{n+1})$ generate $V_{u_t}(Z_t)$ and there is exactly one linear dependence between them $ \sum_{1 \leq i \leq n+1} a_i f'(p_i) = 0$, where $a_i \neq 0$ for all $1 \leq i \leq n+1$.

\medskip

\emph{We will again deform the analytic type of $\tX_t$ in the following in the similar was as before:}

\medskip

Blow up $E_{u_t}$ in $n+1$ generic points $p_1, \cdots, p_{n+1}$, and let the new exceptional divisors be $E_{v_1}, \cdots E_{v_{n+1}}$.
Denote the subsingularity of this blown up singularity supported by the vertices $\calv \cup u_1, \cdots u_t$ by $\tX_u$. We have again $p_g(Z_u) = p_g(Z_b)$, which means in particular that $(R_u(\pi^*(l'_m)), \mfl_s)$ is relative dominant on the cycle $Z_u$.

In the following fix the singularity $\tX_u$ and deform the analytic type of $\tX_b$ by changing the glueing with the tubular neighbourhoods of the exceptional divisors $E_{v_1}, \cdots, E_{v_{n+1}}$ and moving the intersection points $E_{v_1} \cap E_{u_t}, \cdots, E_{v_{n+1}} \cap E_{u_t}$ such that the analytic type of $\tX_b$ remains relatively generic.

We know that $\pi^*(l'_m)$ has got positive coefficents on the vertices $v_1, \cdots v_{n+1}$, and $$\dim(\im(c^{-(n+1) E_{u_t}^*}(Z_t))) = e_{Z_t}(u_t).$$
In exactly the same way as in lemma \ref{coveropen} we get that if we move the intersection points $E_{v_j} \cap E_{u_t}$, then the line bundles $r_u(\pi^*(\calL))$ cover an open set in $ r_s^{-1}(\mfl_s) \subset \pic^{R_u(\pi^*(l'_m))}(\tX_u)$.

It means that for a generic choice of the interection points $p_1', \cdots p_{n+1}'$, the line bundle $r_u(\pi^*(\calL))$ is a generic line bundle in $ r_s^{-1}(\mfl_s) \subset \pic^{\pi^*(l'_m)}(\tX_u)$, which means that $H^0(Z_u, r_u(\pi^*(\calL)))_{reg} \neq \emptyset$.
Consider a section $s \in H^0(Z_u, r_u(\pi^*(\calL)))_{reg}$, and its divisor $D$ on $Z_u$.
We know that $p_1', \cdots, p'_{n+1}$ are generic points of the divisor $E_{u_t}$, which means that we have a linear dependence $$ \sum_{1 \leq i \leq n+1} a_i f'(p_i) = 0,$$ where $a_i \neq 0$ for all $1 \leq i \leq n+1$.

\medskip

Notice that the map $c^{-(n+1) E_{u_t}^*}(Z_t) : \eca^{-(n+1) E_{u_t}^*}(Z_t) \to \pic^{-(n+1) E_{u_t}^*}(Z_t)$ can be identified with the map $c^{-(n+1) E_{u_t}^*}(Z_u) : \eca^{-(n+1) E_{u_t}^*}(Z_u) \to \pic^{-(n+1) E_{u_t}^*}(Z_u)$.
The map $c^{-(n+1) E_{u_t}^*}(Z_u) : \eca^{-(n+1) E_{u_t}^*}(Z_u) \to \pic^{-(n+1) E_{u_t}^*}(Z_u)$ is a submersion in $(p_1', \cdots p_{n+1}')$. This means by the implicit function theorem that we can move the intersection points, such that their derivatives don't vanish, but the line bundle $ r_u(\pi^*(\calL))$ stays the same.

It means that for a generic choice of $ (p_1', \cdots p_{n+1}')$ one has a section in $s \in H^0(Z_u, r_u(\pi^*(\calL)))_{reg}$, whose divisor is $D$ and it is disjoint from the contact points $(p_1', \cdots, p_{n+1}')$.

Since $h^1(\calO_{Z_b}) = h^1(\calO_{Z_u})$ we can find a section $s' \in H^0(Z_b, \pi^*(\calL))_{reg} \neq \emptyset$ with the same divisor $D$.
It means that we have $H^0(Z_b, \pi^*(\calL))_{reg} \neq \emptyset$ and by blowing down we get $H^0(Z_t, \pi^*(\calL))_{reg} \neq \emptyset$. This finishes the proof of part 1) completely.

\bigskip

\textbf{Proof of part 2):}

Consider again the blown up singularity with vertex set $\calv \cup u_0, \cdots, u_t$ and its subsingularity $\tX_s$ with resolution graph supported by the vertex set $\calv \setminus u_t \cup  u_0, \cdots, u_{t-1}$.
We know that $h^1(\calO_{Z_s}) - h^1(\calO_{Z_1}) <  h^1(\calO_Z) - h^1(\calO_{Z_1})$. It means by the induction hypothesis that if $A \leq Z_s$ is any cycle on $\tX_s$ and $\calL_s$ is any natural line bundle on $\tX_t$ satisfiying the conditions of the Theorem, then $h^1(A, \calL_s) = h^1(A, \calL_{gen})$, where $\calL_{gen}$ is a generic line bundle in $r^{-1}( \calL_s | A_1)$ with $A_1 = \min(A, Z_1)$.

\medskip

We prove first the following lemma:

\begin{lemma}\label{segedhegy}
Suppose that $Z' \leq Z_t$ is a cycle on $\tX_t$ and $\calL = \calO_{\tX_t}(l'_t)$ is a natural line bundle on $\tX_t$, such that $l'_t = \sum_{v \in \calv_t} a_v E_v$ with $a_{u_t} < 0$.
With these assumptions we have $h^1(Z', \calL) = h^1(Z', \calL_{gen})$, where $\calL_{gen}$ is a generic line bundle in $r_s^{-1}( \calL | Z'_s )$.
\end{lemma}
\begin{proof}
We prove the lemma by induction on $\sum_{v \in \calv_t} Z'_v $, if this is $0$, then the satement is trivial.
Let us denote $\calL | Z'_s =  \mfl_s$, there are two cases in the following:

\medskip

\underline{Suppose first that $(l'_t,  \mfl_s)$ is not relative dominant on $Z'$:}

\medskip

We know that $h^1(\calO_{Z'}) - h^1(\calO_{Z_s}) \leq  h^1(\calO_Z) - h^1(\calO_{Z_1})$ so we can use the already proven part 1) to get $H^0(Z' , \calL)_{reg} = \emptyset$.
On the other hand we have also $H^0(Z' , \calL_{gen})_{reg} = \emptyset$ which means that:

\begin{equation*}
h^0(Z' , \calL) = \max_{v \in |Z'|}(Z' -E_v, \calL - E_v)  = \max_{v \in |Z'|}(Z' -E_v, \calL_{gen} - E_v) = h^0(Z' , \calL_{gen}) ,
\end{equation*}
by the induction hypothesis, so we are done in this case.

\medskip

\underline{Suppose in the following that  $(l'_t,  \mfl_s)$ is relative dominant on $Z'$:}

\medskip

It means by part 1), that $H^0(Z' , \calL)_{reg} \neq \emptyset$ and $H^0(Z' , \calL_{gen})_{reg} \neq \emptyset$.

\medskip

We will use again our usual deformation technique.
Assume that $N$ is a large number, such that $\dim(\im(c^{-N E_{u_t}^*})(Z')) = e_{Z'}( u_t)$ and blow up $E_{u_t}$ in $N$ generic points, and let the new singularity be $\tX_b$, and let the new vertices be $v_1, \cdots v_N$.
Let us denote the intersection points $p_j = E_{v_j} \cap E_{u_t}$ and the subsingularity supported on the vertex set $\calv \cup u_1, \cdots, u_t$ by $\tX_u$.
Furthermore denote the pullback of the cycle $Z'$ by $Z'_b$ and its restriction to $\tX_u$ by $Z'_u$.
Since we have blown up $E_{u_t}$ in generic points we know that $H^0(Z'_b, \pi^*(\calL))_{reg} \neq \emptyset$.

In the following fix the singularity $\tX_u$ and deform the analytic type of $\tX_b$ by moving the intersection points $E_{v_1} \cap E_{u_t}, \cdots, E_{v_{N}} \cap E_{u_t}$ such that the analytic type of $\tX_b$ remains relatively generic.
We know that $\pi^*(l'_t)$ has got positive coefficents on the vertices $v_1, \cdots v_{N}$, and $\dim(\im(c^{-N E_{u_t}^*}(Z_t))) = e_{Z_t}(u_t)$. We get exactly the same way
as in lemma\ref{coveropen} that if we move the intersection points $E_{v_j} \cap E_{u_t}$, then the line bundles $r_u(\pi^*(\calL))$ cover an open set in $ r_s^{-1}(\mfl_s) \subset \pic^{R_u(\pi^*(l'_t))}(Z'_u)$.

It means that for a generic choice of the contact points one has:

\begin{equation*}
h^1(Z'_u,  \pi^*(\calL)| Z'_u) = h^1(Z'_u, \pi^*(\calL_{gen})| Z'_u)
\end{equation*}

We have $h^1(\calO_{Z'_u}) = h^1(\calO_{Z'_b})$ and $H^0(Z'_b , \pi^*(\calL))_{reg} \neq \emptyset$, which means that $h^1(Z'_b ,  \pi^*(\calL)) = h^1(Z'_u,  \pi^*(\calL)| Z'_u) $.

Similarly we have $h^1(Z'_b , \pi^{*}(\calL_{gen})) = h^1(Z'_u, \pi^{*}(\calL_{gen})| Z'_u) $, which means that $$h^1(Z'_b , \pi^*(\calL)) =h^1(Z'_b , \calL_{gen}).$$
It means that we also have $h^1(Z', \calL) = h^1(Z' , \calL_{gen})$ and the lemma is proved.

\end{proof}

We return to the proof of part 2) in the following:

\medskip

We have to prove that $h^1(Z, \calL) = h^1(Z, \calL'_{gen})$, where $\calL'_{gen}$ is a generic line bundle in $r^{-1}(\mfl) \subset \pic^{l'_m}(Z)$.
Notice that we have proved in Lemma \ref{segedhegy} that $h^1(Z_t, \pi^*(\calL)) = h^1(Z_t, \calL_{gen})$, where $\calL_{gen}$ is a generic line bundle in $r_s^{-1}(  \pi^*(\calL) | Z_s )$.
Denote $ \mfl_s = \pi^*(\calL) | Z_s $, notice that by theorem \ref{th:hegy2rel} we have:

\begin{equation*}
h^1(Z_t, \calL_{gen}) = \chi(\pi^*(-l'_m)) - \min_{0 \leq l \leq Z_t} (- h^1((Z_t - l)_s,   \mfl_s(-l) ) + \chi(-\pi^*(l'_m) + l) ).
\end{equation*}

Consider a generic line bundle $ \mfl_{gen, s} \in r^{-1}(\calL | Z'_1  ) \subset \pic^{R_s(\pi^*(l'_m))}(Z_s)$, we can choose a generic line bundle $\calL''_{gen} \in \pic^{\pi^*(l'_m)}(Z_t)$ as a generic line bundle in $r_s^{-1}(\mfl_{gen, s}) \subset \pic^{\pi^*(l'_m)}(Z_t)$.

Notice that by theorem \ref{th:hegy2rel} we have:

\begin{equation*}
h^1(Z_t, \calL''_{gen}) = \chi(\pi^*(-l'_m)) - \min_{0 \leq l \leq Z_t} (- h^1((Z_t - l)_s,   \mfl_{gen, s}(-l) ) + \chi(-\pi^*(l'_m) + l) ).
\end{equation*}

If we apply the induction hypothesis on $Z_s$, we get that for all cycles $0 \leq l \leq Z_t$ one has $h^1((Z_t - l)_s,   \mfl_{gen, s}(-l) ) = h^1((Z_t - l)_s,  \mfl_s(-l) )$,
which indeed means $h^1(Z_t, \calL_{gen}) = h^1(Z_t, \calL''_{gen})$.

It means that $h^1(Z, \calL) = h^1(Z_t, \calL''_{gen})$ and by blowing down the right hand side we get that $ h^1(Z, \calL) =  h^1(Z, \calL'_{gen})$, where $\calL'_{gen}$ is a generic line bundle in $ r^{-1}(\mfl) \subset \pic^{l'_m}(Z)$, and the proof of part 2) is finished.

\end{proof}

\begin{remark}
We could also prove the previous theorem by using a similar technique as in the proof of the following Theorem\ref{moreg}.
However we preferred the previous proof, as it could be done in a purely combinatorial way and is a prototype of proofs for other problems regarding general surface singularities.
For example we determine the base point structure of natural line bundles in section 7 with the same method.
\end{remark}

\begin{proof}[Proof of Theorem \textbf{B}]

Consider an effective cycle $Z \geq E$ on $\tX$ and write $Z = Z_1 + Z_2$, where $|Z_1| \subset \calv_1$ and $|Z_2| \subset \calv_2$, then we claim the following formula for the cohomology of the cycle $Z$ first:

\begin{equation}\label{Zcohom}
h^1(\calO_Z) = 1 - \min_{E \leq l \leq Z}(\chi(l)-   h^1(\calO_{(Z-l)_1}(-l) )).                                 
\end{equation}

Consider the following exact sequence:

\begin{equation*}
0 \to H^0(\calO_{Z-E}(-E)) \to H^0(\calO_{Z}) \to H^0(\calO_{E}) \to H^1(\calO_{Z-E}(-E)) \to H^0(\calO_{Z}).
\end{equation*}

\medskip

Since the map $H^0(\calO_{Z}) \to H^0(\calO_{E})$ is clearly surjective, we have $h^1(\calO_Z) =  h^1(\calO_{Z-E}(-E))$.  On the other hand part 2) of our main Theorem\ref{relgen} can be applied for the line $\calO_{Z-E}(-E)$, since it certainly satisfies the assumptions of the theorem.

We get from it equation\ref{Zcohom} immediately.

Notice that from the formal neighborhood theorem we know that if $Z \gg 0$ is a very large cycle, then $p_g(\tX) = h^1(\calO_Z)$, from which we indeed get:

\begin{equation*}
p_g(\tX) = 1 - \min_{E \leq l}(\chi(l)-   h^1(\calO_{\tX_1}(-l) )).                                 
\end{equation*}

This proves indeed Theorem \textbf{B} completely.

\end{proof}

\bigskip

\subsection{Analytic Poincaré series and analytic semigroup}

In this subsection we have the setup as before, so consider two rational homology sphere resolution graphs $\mathcal{T}_1 \subset \mathcal{T} $, a fixed singularity $\tX_1$ for the subgraph $\mathcal{T}_1$, and cuts $D_{v_2}$.
Assume furthermore that $\tX$ has a relatively generic analytic stucture on $\mathcal{T}$ corresponding to $\tX_1$ and the cuts $D_{v_2}$.

In the following corollary we determine the analytic Poincaré series.
Notice that by \cite{NPS} it is equivalent to determine the cohomology numbers of all natural line bundles on the singularity $\tX$ and they can be described more naturally in the current situation (for more about the analytic Poincaré series see \cite{CDGPs}, \cite{CHR}).

\begin{corollary}

a) If $l' \in L'$ is a Chern class such that $l' \notin L$, then we have $h^1(\calO_{\tX}(l')) = \chi(-l')  - \min_{0 \leq l}(\chi(-l' + l)-   h^1(\calO_{\tX_1}(l'-l) ))$.

\medskip

b)  If $l' \in L$ is a Chern class such that $l' \ngeq 0$ then $h^1(\calO_{\tX}(l')) = \chi(-l')  - \min_{0 \leq l}(\chi(-l' + l)-   h^1(\calO_{\tX_1}(l'-l) ))$.

\medskip

c)  If $l' \in L$ is a Chern class such that $l' \geq 0$ then $h^1(\calO_{\tX}(l')) = \chi(-l')  - \min_{0 \leq l}(\chi(-l' + l)-   h^1(\calO_{\tX_1}(l'-l) )) + D$, where $D = 0$ if the pair $(0, \tX_1)$ is relative dominant on $\tX$ (which is equivalent to $p_g(\tX) = p_g(\tX_1)$) and $D= 1$ otherwise.
\end{corollary}

\begin{proof}

For part a) assume first that $l' \notin L$.
Let $s(-l') \in S'$ be the smallest element in the Lipman cone such that  $0 \leq s(-l') + l' \in L$, then we have $h^1(\calO_{\tX}(l')) = \chi(-l') - \chi( s(-l')) + h^1(\calO_{\tX}(s(-l')))$.
On the other hand the line bundle $\calO_{\tX}(s(-l'))$ satisfies the conditions of our main Theorem\ref{relgen} because $s(-l') \neq 0$ and it has got positive coefficient, so we get 
$h^1(\calO_{\tX}(s(-l')))  = \chi(s(-l'))  - \min_{0 \leq l}(\chi(s(-l') + l)-   h^1(\calO_{\tX_1}(s(-l')-l) ))$.
Furthermore it can easily be seen that $\min_{0 \leq l}(\chi(-l' + l)-   h^1(\calO_{\tX_1}(l'-l) )) = \min_{0 \leq l}(\chi(s(-l') + l)- h^1(\calO_{\tX_1}(s(-l')-l) ))$ since
a pair $(l'-l, \calO_{\tX_1}(l'-l) ))$ can only be relative dominant on $\tX$ if $l' - l \in - S'$, so it proves the statement in the case $l' \notin L$.

\medskip

The proof of part b) is exactly the same.

\medskip

For part c) notice that if $l' \in L$ and $l' \geq 0$ then we have $h^1(\calO_{\tX}(l')) = \chi(-l') - \chi(0) + p_g(\tX)$.
If the pair $(0, \tX_1)$ is relative domiant on $\tX$ then $p_g(\tX) = - \min_{0 \leq l}(\chi(l)-   h^1(\calO_{\tX_1}(-l) ))$ and if the pair $(0, \tX_1)$ is not relative domiant on $\tX$ then 
$p_g(\tX) = 1 - \min_{0 \leq l}(\chi(l)-   h^1(\calO_{\tX_1}(-l) ))$, the statement follows from this immediately.
\end{proof}

In the following corollary we determine the analytic semigroup of the relatively generic singularity $\tX$:

\begin{corollary}
Assume that we have a Chern class $l' \in L'$, then we have $l' \in S'_{an}$ if and only if $l' = 0$, or the pair $(l', \calO_{\tX_1}(l'))$ is relative dominant on $\tX$, or equivalently:

\begin{equation*}
\chi(-l') - h^1(\calO_{\tX_1}(l')) < \chi(-l' + l) - h^1(\calO_{\tX_1}(l'-l)),
\end{equation*}
where $0 < l \in L$.
\end{corollary}

\begin{proof}
Notice first that $0$ is certainly in $S'_{an}$ and if $l' \notin S'$ then $l' \notin S'_{an}$ and also the pair $(-l', \calO_{\tX_1}(-l'))$ cannot be relative dominant on $\tX$.

On the other hand if $0 \neq l' \in S'$, then all the coeficcients of $l'$ are positive so by our main Theorem \ref{relgen} we know that $l' \in S'_{an}$ if and only if $H^0(\calO_{\tX_1}(-l'))_{reg} \neq \emptyset$ if and only if the pair $(l', \calO_{\tX_1}(l'))$ is relative dominant on $\tX$. This proves the corollary completely.
\end{proof}
\begin{remark}
Note that in particular $Z_{max}$ is the minimal element in $L_{>0}$ such that 
\begin{equation*}
\chi(Z_{max}) - h^1(\calO_{\tX_1}(-Z_{max})) < \chi(Z_{max} + l) - h^1(\calO_{\tX_1}(-Z_{max}-l)),
\end{equation*}
for every $0 < l \in L$.
\end{remark}

\subsection{Sharpening of the main theorem}

The following theorem is in some way sharper then Theorem\ref{relgen}, and we will need it for technical reasons in following manusripts.

\medskip

The setup will be similar as above:

\medskip

Consider two rational homology sphere resolution graphs $\mathcal{T}_1 \subset \mathcal{T} $, a fixed singularity $\tX_1$ for the subgraph $\mathcal{T}_1$, and cuts $D_{v_2}$.
Assume that we have an effective cycle $Z$ and write $Z = Z_1 + Z_2$, where $|Z_1| \subset \calv_1$ and $|Z_2| \subset \calv_2$.

With this setup we have the following theorem:

\begin{theorem}\label{moreg}

Assume that $\tX$ has a relatively generic analytic stucture on $\mathcal{T}$ corresponding to $\tX_1$ and the cuts $D_{v_2}$.
Consider a Chern class $ l' = - \sum_{v \in \calv} a_v E_v$  and assume that  $a_v \neq 0$ if $v \in \calv_2 \cap |Z|$.
For the natural line bundle $\calL= \calO_{\tX}(l')$ denote $c_1 (\calL | Z) = l'_ m \in  L'_{|Z|}$ and $\mfl = \calL | Z_1$.

\medskip

Assume that $H^0(Z,\calL)_{reg} \not=\emptyset$, and pick an arbitrary divisor $D \in c^{l'_m}(Z)^{-1}(\calL) \subset \eca^{l'_m, \mfl}(Z)$.

\medskip

1) Under these conditions the map $ c^{l'_m}(Z) : \eca^{l'_m, \mfl}(Z) \to r^ {-1}(\mfl  )$ is a submersion in $D$, and $h^1(Z,\calL) = h^1(Z_1, \mfl)$.

2) In particular the map $ c^{l'_m}(Z) :  \eca^{l'_m, \mfl}(Z) \to r^ {-1}(\mfl  )$ is dominant, which means $(l'_m,\mfl)$ is relative dominant on the cycle $Z$, or equivalently:

\begin{equation*}
\chi(-l')- h^1(Z_1, \mfl) < \chi(-l'+l)-  h^1((Z-l)_1, \mfl(-l)),
\end{equation*}
for all $0 < l \leq Z$.

\end{theorem}

\begin{remark}
The proof will follow the proof of the analouge Theorem in the nonrelative setup from \cite{NNA2} with basically just technical modifications.
\end{remark}

\begin{proof}

\textbf{For part 1)} assume to the contrary that $D \in c^{l'_m}(Z)^{-1}\calL \subset \eca^{l'_m, \mfl}(Z)$, but $ c^{l'_m}(Z) : \eca^{l'_m, \mfl}(Z) \to r^ {-1}(\mfl  )$ is not a submersion in $D$. It means that there is an element $w \in  (r^{-1}( \mfl))^{*}$ , such that $d(w \circ c^{l'_m}(Z)) $ vanishes in $D \in  \eca^{l'_m, \mfl}(Z)$, where $(r^{-1}( \mfl))^{*}$ means the dual of the linearisation of the affine space  $r^{-1}(\mfl)$.  An element $w \in  (r^{-1}( \mfl))^{*}$ gives an affine function $w : r^{-1}(\mfl) \to \bC$ uniqely up to an additive constant.
The map $w : r^{-1}(\mfl) \to \bC$ defines also maps $V_I(Z) \to \bC$ for all $I \subset \calv_2$, we denote all of them by $w$ for the sake of simpleness of the notations.

Notice that $H^1(\calO_Z)^*$ can be identified with the differential forms $\frac{H^0(\calO_{\tX}( K + Z)) }{H^0(\calO_{\tX}(K))} \subset H^1(\calO_{\tX})^*$ having pole on a vertex $E_v$ of order at most $Z_v$ for every $v \in \calv$.
Denote by $\Omega_{Z, \calv_2} \subset \frac{H^0(\calO_{\tX}( K + Z)) }{H^0(\calO_{\tX}(K))}$ the differential forms, which have got no pole on the exceptional divisors corresponding to the vertices $v_2 \in \calv_2$ and let $W \subset \frac{H^0(\calO_{\tX}( K + Z)) }{H^0(\calO_{\tX}(K))}$ be a complementary subspace of $\Omega_{Z, \calv_2}$.
The subspace of differential forms $W$ can be identified with the space $ (r^{-1}( \mfl))^{*}$, and we can see $w$ as a differential form $w \in W$.

Let us denote by $A$ the affine clousure operator in the following.
We know that there is a vertex $v \in \calv_2$, such that $w$ has got a pole on the exceptional divisor $E_v$. It means by \cite{NNA1} that the map $  \eca^{-E_v^*}(Z) \to A( \im(  c^{-E_v^*}(Z))) \to \bC$ is non constant, where the second map is $A( \im(  c^{-E_v^*}(Z))) \to V_Z(v) \to \bC$, which is an arbitrary translation $A( \im(  c^{-E_v^*}(Z))) \to V_Z(v)$ composed by the linear form $w$.

We blow up the exceptional divisor $E_v$ and sequentially the new exceptional divisors in generic points. 
Let the new exceptional divisors be $ E_{v_0} = E_v, E_{v_1},...E_{v_t}$ and denote by $t$ the minimal number such that $w$ hasn't got a pole along the exceptional divisor $E_{v_t}$, it means in particular that $w$ has got a pole on $E_{v_{t-1}}$ of order $1$ (since we are blowing up the exceptional divisors in generic points, the order of the pole of $w$ decreases by $1$ at each step). We know that $t$ must be finite, in fact $t \leq Z_v$, because the order of the pole of $w$ decreases by $1$ at each blow up. We also have obviously $t \geq 1$.

Denote the blown up singularities by $\tX_{b, i}$ and the vertex sets and resolution graphs of $\tX_{b, i}$ by $\calv_{b,i}, \mathcal{T}_i$.
Consider the cycles on $\tX_{b, i}$, $Z_{b, i} = \pi_i^*(Z) $. We know that $H^1(\calO_{Z_{b,i}}) = H^1(\calO_Z)$, and we have the restriction map $r_i : H^1(\calO_{Z_{b, i}}) \to H^1(\calO_{Z_1})$, and $w$ gives unique maps $r_i^{-1}(\mfl) \to \bC$ (up to additive constant).
The map $r_i^{-1}(\mfl) \to \bC$ defines also linear maps $V_{Z_{b, i}}(I) \to \bC$ for all subsets $I \subset \calv_{b, i} \setminus \calv_1$, we still denote them by $w$.
We know $\eca^{-E_{v_t}^*}(Z_{b, t}) \to A(\im(c^{-E_{v_t}^*}(Z_{b, t}) )) \to \bC$ is constant, since $w$ hasn't got a pole on the exceptional divisor $E_{v_t}$.

Similarly the differential form $w$ has got a pole on the exceptional divisor $E_{v_{t-1}}$ of order $1$, so the map $\eca^{-E_{v_{t-1}}^*}(Z_{b, t-1}) \to A(\im(c^{-E_{v_{t-1}}^*}(Z_{b, t-1}) )) \to \bC$ depends just on the support point of the divisor, and gives a function $f : \bP - \delta_{t-1} \cdot points \to \bC$.
We blow up $E_{v_{t-1}}$ in a generic point $p$, so we can assume that $w$ has not got an arrow (clousure of set of vanishing of the differential form $w$ in $\tX_{t-1} \setminus E_{v_{t-1}}$) at $p$. It means by Laufer integration formula to that $f$ has non zero differential at $p$.

Consider the cycle $Z'_{t-1}= Z_{b,t} - Z_v \cdot E_{v_t} $ on $\tX_{b, t}$, and let $r' : H^1(\calO_{Z'_{t-1}}) \to  H^1(\calO_{Z_1})$ be the restriction map.

We have the surjective map $H^1(\calO_{Z_{b, t}}) \to H^1(\calO_{Z'_{t-1}})$ and the injective dual map $$\frac{H^0(\calO_{\tX_t}( K_t + Z'_{t-1})) }{H^0(\calO_{\tX_t}( K_t))} \to \frac{H^0(\calO_{\tX_t}(K +Z_{b, t})) }{H^0(\calO_{\tX_t}(K_t))}.$$ Since $w$ has not got a pole along the exceptional divisor $E_{v_t}$, we get that $w \in \frac{H^0(\calO_{\tX_t}( K_t + Z'_{t-1})) }{H^0(\calO_{\tX_t}( K_t))}$.

It means that if we have a line bundle $\calL' \in  \pic^{0}(Z_{b, t}) \cong H^1(\calO_{Z_{b, t}})$, then the value of $w$ on the line bundle $\calL'$  depends only on the restriction $\calL' | Z'_{t-1}$.

\medskip

\emph{We can define a function $\eta: \eca^{-E_{v_{t-1}}^*}(Z'_{t-1}) \to \bC$ in the following:}

\medskip

Fix a divisor $D' \in \eca^{-E_{v_{t-1}}^*}(Z'_{t-1})$ and for any divisor $D'' \in \eca^{-E_{v_{t-1}}^*}(Z'_{t-1})$ define $$\eta(D'') = w(\calO_{Z'_{t-1}}(D''- D')).$$

Since the differential form $w$ has got a pole along the exceptional divisor $E_{v_{t-1}}$ of order $1$, we can easily see that the map $\eta$ depends only on the support of the divisor $D''$ and it gives a map $g :  \bP - \delta_{t-1} \cdot points \to \bC$, where $ g = f +  c$ for some $c \in \bC$.
Indeed if the support of the divisor in $\eca^{-E_{v_t}^*}(Z_{b, t})$ is not $p$, then the statement is trivial by pushing it down to $Z_{t-1}$ and it follows in $p$ by continuity, so we get that the derivative of $g$ doesn't vanish in $p$.

\medskip

\emph{Similarly we can define a function $h: \eca^{R(\pi^*(l'_m)), \mfl}(Z'_{t-1}) \to \bC$ by the same manner:}
\medskip
Fix a divisor $D^* \in \eca^{R(\pi^*(l'_m)), \mfl}(Z'_{t-1})$ and for any divisor $D'' \in \eca^{R(\pi^*(l'_m)), \mfl}(Z'_{t-1})$ define $h(D'') = w(\calO_{Z'_{t-1}}(D''- D^*))$.
It is clear from our condition on $w$, that the map $h$ has derivative $0$ in $D$.

\medskip

Next we will use the relative genericity of the resolution $\tX$ by deforming $\tX$ in a way such that $\tX_1$ stays the same.
Define a germ $S = (\bC, 0)$ of singularities $\tX_u$ by fixing the tubular neighborhood of the exceptional divisors $E_v, v \in \calv, E_{v_1}, ..., E_{v_{t-1}}$ and we change the plumbing with the tubular neighborhood of $ E_{v_{t}}$ such that the derivative of the intersection point $ E_{v_{t-1}} \cap E_{v_{t}}$ doesn't vanish. 
An explicit formula for such a deformation is described in \cite{NNA2}.

Let us denote by $\calL_{u} = \calO_{\tX_u}(\pi^*(l'))$ the natural line bundle on $\tX_u$ with the same Chern class as $\pi^*(\calL)$.
If the original singularity was relatively generic with respect to $\tX_1$ and the cuts $D_{v_2}$, then the combinatorial type of the subspace complement $H^0(Z_u, \calL_u)_{reg}$ remains stable and rises to a fibration over $S$. It means that if $u$ is small enough, there is a family of divisors $D_u \in \eca^{R(\pi^*(l'_m)), \mfl}(Z'_{t-1})$, such that $D_u \in H^0(Z_u, \calL_u)_{reg}$ and $D_0 = D$.

Notice that if $\calL'$ is a line bundle on $\tX_u$ with Chern class $\pi^*(l')$ we can still speak about the value of $w$ on $\calL' $, by defining it as 
$w(\calL') = w(\calL'|Z'_{t-1}  \otimes \calO_{Z'_{t-1}}(-D^*))$.
Notice that we have $\frac{d}{du}(h(D_u)) = 0$.

On the other hand we have 
$$\calL_u = \calO_{Z_u}(-\sum_{s \in \calv}a_s E_s - \sum_{1 \leq j \leq t-1}a_v E_{v_j}- a_v E_{v_t, u}).$$

This means that we have the following equation:

\begin{equation*}
\frac{d}{du}(h(D_u)) = - a_v \cdot \frac{d}{du}(  f(E_{v_t, u} \cap E_{v_{t-1}})).
\end{equation*}

On the other hand we know that $df \neq 0$ in $p$, this is a contradiction which indeed proves part 1) completely.

\bigskip

\textbf{For part 2)} notice that since $ c^{l'_m}(Z) : \eca^{l'_m, \mfl}(Z) \to r^ {-1}(\mfl  )$ is a submersion in $D$, it is trivial that $(l'_m,\mfl)$ is relative dominant on the cycle $Z$, which is euqivalent to:

\begin{equation*}
\chi(-l'_m)- h^1(Z_1, \mfl) < \chi(-l'_m+l)-  h^1((Z-l)_1, \mfl(-l)),
\end{equation*}
for all cycles $0 < l \leq Z$.

We know that  $ c^{l'_m}(Z) : \eca^{l'_m, \mfl}(Z) \to r^ {-1}(\mfl  )$ is a submersion in $D$, which means that $\im( T_{D}(c^{l'_m}(Z)))$ contains $T_{c^{l'_m}(D)}r^ {-1}(\mfl  )$. On the other hand we know that $ \im(T_{D}( r \circ c^{l'_m}(Z)))$ has codimension $h^1(Z_1, \mfl)$ in $H^1(\calO_{Z_1})$, which means
$\im( T_{D}(c^{l'_m}(Z)))$ has codimension $h^1(Z_1, \mfl)$ in $H^1(\calO_Z)$. This gives the desired equality $h^1(Z, \calL) = h^1(Z_1, \mfl)$.

\end{proof}

\section{Base points of natural line bundles on relatively generic surface singularities}

In this section we wish to determine the base point structure of natural line bundles on relatively generic surface singularities.
We will consider the bit more general case of restricted natural line bundles for some technical reasons.

\medskip

Fix the setup in the following:

\medskip 

Consider three rational homology sphere resolution graphs $\mathcal{T}_1 \subset \mathcal{T} \subset \mathcal{T}'$ and a fixed singularity $\tX_1$ corresponding to the resolution graph $\mathcal{T}_1$ and for every vertex $v \in \calv'$ which has a neighbour in $\calv$ a cut $D_v$.
Let $\tX_1$ be a fixed singularity corresponding to $\mathcal{T}_1$ and assume that $\tX'$ is a relatively generic singularity with resolution graph $ \mathcal{T}'$  with respect to $\tX_1$ and the cuts  $D_v$.
Fix furthermore an element $l'_{top} \in L'_{\mathcal{T}'}$, such that if $v \in \calv$ we have $(l'_{top})_v > 0$. Let us denote $ R(l'_{top})= l'_d  \in L'_{\mathcal{T}}$ and $\calL = \calO_{\tX_1}(- l'_{top})$,  $\calL_d = \calO_{\tX}(- l'_{top})$.
Assume that the pair $(-l_d, \calO_{\tX_1}(-l'_{top}))$ is relatively dominant. This means by Theorem\ref{relgen} that $H^0(\tX, \calL_d)_{reg} \neq \emptyset$ and $h^1(\tX_1,  \calL) = h^1(\tX, \calL_d)$.

\begin{theorem}\label{basepointrel}
Assume that the line bundle $\calL$ has no base points on $\tX_1$, then:

\medskip

1) The line bundle $\calL_d$ has no base points at the intersection points of exceptional divisors, and it has got only simple base points.

2) If $v \in \calv$, then $ \calL_d$ has base point along the exceptional divisor $E_v$ if and only if there exists a cycle $A \geq E_v$, such that:

\begin{equation*}
\chi(l'_d) - h^1(\tX_1, \calL) + 1 = \chi(l'_d + A) - h^1(\tX_1, \calL( -A)).
\end{equation*}

Furthermore if there is a base point on $E_v$, then there are $-(l'_d, E_v)$ distinct base points along the exceptional divisor $E_v$.
\end{theorem}
\begin{remark}
Theorem \textbf{C} follows immediately, if we apply the theorem in the situation when $\mathcal{T} = \mathcal{T}'$ and $\calL_d = \calO_{\tX}(- Z_{max})$.
\end{remark}

\begin{proof}

Notice first that the pair $(-l'_d, \calO_{\tX_1}(-l'_{top}))$ is relatively dominant, which means by Theorem\ref{th:dominantrel} that:

\begin{equation*}
\chi(l'_d) - h^1(\tX_1, \calL) < \chi(l'_d + l) - h^1(\tX_1, \calL( -l) ),
\end{equation*}
for all $0 < l  \in L_{\mathcal{T}}$.

\medskip

We will prove statements 1) and 2), by a simultaneous induction on the parameter $ p_g(\tX) - p_g(\tX_1)$.

\bigskip

\textbf{Proof in the base case $ p_g(\tX) = p_g(\tX_1)$:}

\medskip
In this case since the line bundle $\calL$ has no base points on $\tX_1$, we get that $\calL_m$ hasn't got also any base points.

Indeed if $s \in H^0(\tX_1, \calL)_{reg}$ is a generic section and $D = |s|$, then if we have a divisor $D'$ on $\tX$ such that the restriction of $D'$ equals $D$, then by the condition
$ p_g(\tX) = p_g(\tX_1)$ we get that $\calO_{\tX}(D') = \calL_d$.
Since $\calL$ has no base points, this proves indeed that the line bundle $\calL_d$ hasn't got any base points too.

We only need to prove in the following that if $-(l'_d, E_v) > 0$ and $A > 0 , A \in L_{\mathcal{T}}$, such that $E_v \in |A|$, then one has:

\begin{equation*}
\chi(l'_d) - h^1(\tX_1, \calL) + 1 < \chi(l'_d + A) - h^1(\tX_1, \calL(-A)),
\end{equation*}

Asume to the contrary that:

\begin{equation*}
\chi(l'_d) - h^1(\tX_1, \calL) + 1 = \chi(l'_d + A) - h^1(\tX_1, \calL(-A) ),
\end{equation*}

Consider the unique line bundle $ \calL_{gen} \in \pic^{-l'_d}(\tX)$, such that $r(\calL_{gen}) = \calL$, we know that $ \calL_{gen} = \calL_d$ and it is oviously relatively generic.

We have the following exact sequence:

\begin{equation*}
0 \to H^0(\tX, \calL_{gen}(- E_v)) \to H^0(\tX, \calL_{gen}) \to H^0(E_v, \calL_{gen} ) \to  H^1(\tX, \calL_{gen}(- E_v)) \to H^1(\tX, \calL_{gen}) \to 0.
\end{equation*}

By Theorem \textbf{A} we get:
\begin{equation*}
h^1(\tX, \calL_{gen}(- E_v))  = \chi(l'_d + E_v) - \min_{0 \leq l \in L_{\mathcal{T}}}( \chi(l'_d + E_v+ l) - h^1(\tX_1, \calL(-l - E_v))).
\end{equation*}

\begin{equation*}
h^1(\tX, \calL_{gen}( - E_v))  = \chi(l'_d + E_v) - ( \chi(l'_d + A) - h^1(\tX_1, \calL(-A))) = \chi(l'_d + E_v) - (\chi(l'_d) - h^1(\tX_1, \calL) + 1).
\end{equation*}

\begin{equation*}
h^1(\tX, \calL_{gen}(- E_v)) = -(l'_d, E_v) + h^1(\tX_1, \calL).
\end{equation*}

It means by the exact sequence that the dimension of the image of the map $ H^0(\tX, \calL_{gen}) \to H^0(E_v, \calL_{gen} )$ is $1$, which means that $\calL_{d} = \calL_{gen}$ has got a base point on $E_v$, which is a contradiction. Since there are no base points of the line bundle $\calL_d$, part 2) is obvious. It finishes the proof in the case $p_g(\tX) = p_g(\tX_1)$.

Consider next the induction step, assume that $p_g(\tX) - p_g(\tX_1) = k$ and we know part 1) and 2) in cases when $ p_g(\tX) - p_g(\tX_1) <k$.

\bigskip

\textbf{Proof of the fact that there are no base points at intersection points: }

Assume that the line bundle $\calL_d$ has got a base point at the intersection point of two exceptional divisors $E_v$ and $E_w$, and blow up this intersection point, let us denote the blow up map by $\pi$ and new exceptional divisor by $E_{new}$. Let $\tX_{1, new}, \tX_{new}, \tX'_{new}$ be the new resolutions (notice that $\tX_{1, new} = \tX_1$ if and only if $v, w \notin \calv_1$).

Notice that $\tX'_{new}$ is a relatively generic singularity with respect to $\tX_{1, new}$ and the strict transforms of the original cuts.

\medskip

Notice that it is enogh to prove that the pair $(-\pi^*(l'_d), \calO_{\tX_{1, new}}(-\pi^*(l'_{top})) )$ is relatively dominant on $\tX_{new}$, because it would yield $H^0( \calO_{\tX_{new}}(-\pi^*(l'_{top})) )_{reg} \neq \emptyset$.  However this contradictions to the fact that the line bundle $\calL_d$ has a base point at the intersection point of the two exceptional divisors $E_v$ and $E_w$. Denote $\calL' = \calO_{\tX_{1, new}}(-\pi^*(l'_{top}))$, we have to prove that for all cycles $l \in L$ and nonnegative integer $a \geq 0$, with the condition that if $l = 0$, then $a > 0$, we have:

\begin{equation*}
\chi(\pi^*(l'_d)) - h^1(\tX_{1, new},  \calL') < \chi(\pi^*(l'_d + l) + a E_{new}) - h^1(\tX_{1, new}, \calL' (- \pi^*(l) - a E_{new}) ).
\end{equation*}

\begin{equation*}
\chi(l'_d) -   h^1(\tX_{1}, \calL)< \chi(l'_d + l) + \frac{a(a+1)}{2} -  h^1(\tX_{1, new}, \calL' (-\pi^*(l) - a E_{new})),
\end{equation*}
where we have $\calL' = \calO_{\tX'_{1, new}}(-\pi^*(l'_{top}))$.

\medskip

The conditions for the cycle $l$ and the number $a$ express the facts that $\pi^*(l'_d) + \pi^*(l)+ a E_{new} > \pi^*(l'_d)$ and $\pi^*(l'_d + l) + a E_{new} \in S'_{new}$.

\medskip

\underline{Assume first that $l= 0$ and $a > 0$.}

\medskip

We have two cases in the following:

\medskip

If $v, w \notin \calv_1$, then $ h^1(\tX_{1, new}, \calL' (-\pi^*(l) - a E_{new})) = h^1(\tX_{1}, \calL)$, and we have to prove that:

\begin{equation*}
\chi(l'_d) -   h^1(\tX_{1},  \calL)< \chi(l'_d) + \frac{a(a+1)}{2} -  h^1(\tX_{1}, \calL),
\end{equation*}
which is trivial.

\medskip

In the other case if $v\in \calv_1$ or $w \in \calv_1$, then we have $E_{new} \in \calv_{1, new}$ and we have to prove that:

\begin{equation*}
\chi(l'_d) - h^1(\tX_{1, new},  \calL') < \chi(l'_d) + \frac{a(a+1)}{2} - h^1(\tX_{1, new}, \calL'( - a E_{new})).
\end{equation*}

\begin{equation*}
 - h^1(\tX_{1, new},  \calL') <  \frac{a(a+1)}{2} - h^1(\tX_{1, new}, \calL'(- a E_{new})).
\end{equation*}

We know that $H^0(\tX_{1, new}, \calL'(- a E_{new}) ) \subset  H^0(\tX_{1, new},  \calL') $ and $H^0(\tX_{1, new}, \calL'(- a E_{new}) ) \neq H^0(\tX_{1, new},  \calL') $ because the line bundle $\calL$ does not have a base point on $\tX_1$, from which we get:

\begin{equation*}
 h^1(\tX_{1, new},  \calL')  - \chi ( -c^1(\calL'))   >  h^1(\tX_{1, new}, \calL'(- a E_{new})) - \chi ( -c^1(\calL') + a E_{new}).
\end{equation*}

This exactly gives:

\begin{equation*}
 h^1(\tX_{1, new},  \calL')  + \frac{a(a+1)}{2}  >   h^1(\tX_{1, new}, \calL'(- a E_{new})).
\end{equation*}

\medskip

\underline{Assume in the following that $l> 0$ and $a \geq 0$, in this case we have to prove:}

\medskip

\begin{equation*}
\chi(l'_d) -   h^1(\tX_{1},  \calL)< \chi(l'_d + l) + \frac{a(a+1)}{2} - h^1(\tX_{1, new}, \calL' (-\pi^*(l) - a E_{new}) ).
\end{equation*}

We know that $ (-l'_d, \calL) $ is relatively dominant, which means:

\begin{equation*}
\chi(l'_d) -   h^1(\tX_{1},  \calL)< \chi(l'_d + l) - h^1(\tX_{1}, \calL (-l)).
\end{equation*}

\medskip

It means that we only have to prove $\frac{a(a+1)}{2} - h^1(\tX_{1, new}, \calL' (-\pi^*(l) - a E_{new}) ) \geq  -h^1(\tX_{1}, \calL(-l))$.
We know that  $h^1(\tX_{1}, \calL( -l)) = h^1(\tX_{1, new}, \calL'(- \pi^*(l)))$ and $ H^0(\tX_{1, new}, \calL' (-\pi^*(l) - a E_{new}) ) \subset H^0(\tX_{1, new}, \calL'(-  \pi^*(l)))$, from which we get:

\begin{equation*}
h^1(\tX_{1, new}, \calL'(- \pi^*(l)) ) - \chi(-c^1(\calL' (-\pi^*(l)))) \geq h^1(\tX_{1, new}, \calL' - (\pi^*(l) + aE_{new}) ) - \chi(-c^1(\calL'(-\pi^*(l)  -aE_{new})) ) .
\end{equation*}
This exactly gives:
\begin{equation*}
h^1(\tX_{1, new}, \calL'(- \pi^*(l))) \geq h^1(\tX_{1, new}, \calL' (-\pi^*(l) - aE_{new}) ) - \frac{a(a+1)}{2}.
\end{equation*}

It finishes the prof that the line bundle $\calL_d$ does not have a base point at the intersection point of exceptional divisors. 

\bigskip

\textbf{Proof of the fact that the line bundle $\calL_d$ has got only simple base points:}

\medskip

In the following we want to show that the line bundle $\calL_d $ has only simple base points.
We know that $p_g(\tX) > p_g(\tX_1)$, which means by \cite{NNA1} that there is a vertex $u \in \calv \setminus \calv_1$ such that, there are differential forms in $H^1(\calO_{\tX})^* \cong H^0(\tX\setminus E, \Omega^2_{\tX})/ H^0(\tX,\Omega_{\tX}^2)$, which have got pole on the exceptional divisor $E_u$.

Blow up $E_u$ in generic points sequentially and let the new vertices be $u_0 = u, u_1, \cdots, u_t$. Let $t$ be the maximal positive integer, such that there are nonzero differential forms in $H^1(\calO_{\tX})^* \cong H^0(\tX\setminus E, \Omega^2_{\tX})/ H^0(\tX,\Omega_{\tX}^2)$ having pole on $E_{u_t}$.
Denote the blown up singularities by $\tX_{b, 1} = \tX_1, \tX'_b, \tX_b$ and the subsingularity with resolution graph $(\mathcal{T} \setminus u_t) \cup u_1, \cdots, u_{t-1}$ by $\tX_s$.
Notice that in the case $t = 0$ the singularity $\tX_s$ is not nessecarily connected.

If $t> 0$ then denote the intersection point $E_{u_{t-1}} \cap E_{u_t}$ by $q$ and if $t =0$ then denote the set of intersection points $ E_u \cap (E - E_u)$ by $(q_1, \cdots, q_k)$. 
We know that $p_g (\tX_s) - p_g(\tX_1) < p_g (\tX) - p_g(\tX_1)$ and $\tX_s$ is relatively generic to $\tX_1$, which means that we can use the induction hypothesis for the pair of singularities $(\tX_s, \tX_1)$.

Consider the line bundle $\calL_s = \calO_{\tX_{s}}(- \pi^*(l'_{top}))$, notice that the pair $(-R_s(\pi^*(l'_d)), \calL)$ is relatively dominant on $\tX_s$, since $H^0(\tX_s, \calL_s)_{reg} \neq \emptyset$.  By the induction hypothesis we get that the line bundle $\calL_s$ has only simple base points only at regular points of the exceptional divisors $E_v$, where $v \in \calv_s$.

We also know by the other induction hypothesis that if $\calL_s$ has a base point on $E_v$, where $v \in \calv_s$, then it has $-(l'_d, E_v)$ disjoint base points at regular points of $E_v$.
In this case we get that $\calO_{\tX_{b}}(- \pi^*(l'_{top}))$ or equivalently $ \calL_d $ has $-(l'_d, E_v)$ disjoint base points at regular points of $E_v$, and then trivially they are simple base points.

Thus we can assume that $v \in \calv$ and $\calL_s$ does not have a base point on the exceptional divisor $ E_v$. 
It means in particular that a generic section $s \in H^0(\tX_s, \calL_s)$ has a divisor which has $-(l'_d, E_v)$ dijsoint transversal cuts on the regular part of the exceptional divisor $E_v$.

\medskip

\emph{We will use our usual deformation technique in the following:}

Assume that $N$ is a large number, such that $\dim(\im(c^{-N E_{u_t}^*}))= e_{\tX_b}(u_t)$, and blow up the exceptional divisor $E_{u_t}$ in $N$ distinct generic points $p_1, \cdots p_N$, and let the new exceptional divisors be $E_{w_1}, \cdots , E_{w_N}$.
Denote the new singularity by $\tX_{new}$ and its subsingularity with vertex set $\calv \cup u_1 , \cdots , u_t$ by $\tX_u$. 
We know that $\tX_{new}$ has got a relatively generic analytic structure with respect to $\tX_1$ and the cuts.
We also know similarly as before that $p_g(\tX_u) = p_g(\tX_{new}) = p_g(\tX)$.

It means that $\pic^{-\pi^*(l'_d)}(\tX_{new}) \cong \pic^{-R_u(\pi^*(l'_d))}(\tX_u)$ and the isomorphism map is given by the restriction map $r_u$ ($R_u$ is the corresponding cohomological restriction operator).

Next we will deform the analytic structure of $\tX_{new}$ in such a way that we fix the analytic type of $\tX_1$ and the cuts.

We fix the analytic type of the subsingularity $\tX_u$ and change the plumbing with the tubular neighborhood of the exceptional divisors be $E_{w_1}, \cdots , E_{w_N}$ in such a way
that we move the contact points $p_1, \cdots, p_N$.
Notice that from the definition of the number $t$, we know that the cohomological cycle of $\tX_u$ has got multiplicity one on the vertex $u_t$. It means that if $D'$ is a smooth divisor
intersecting the exceptional divisor $E_{u_t}$ at a regular point transversally then the line bundle $\calO_{\tX_u}(D')$ depends only on the intersection point $D' \cap E_{u_t}$.
We also know that if we move the contact points $p_1, \cdots, p_N$ and the plumbing with the tubular neighborhood of the exceptional divisors $E_{w_1}, \cdots , E_{w_N}$
minorly then the given analytic type on $\tX_{new}$ remains relatively generic.

Notice that $\dim(\im(c^{-N E_{u_t}^*}))= e_{\tX_b}(u_t)$ and the coefficients of $E_{w_1}, \cdots, E_{w_N}$ in $ \pi^*(l'_{top})$ are positive.
It means that exactly the same way as in Lemma \ref{coveropen} we get that  if we move the contact points $p_1, \cdots, p_N$, then the possible values of the restricted line bundle $\calL_u = r_u( \pi^*(\calL_d) ) $ cover an open set in $ r_s^{-1}(\calL_s) \subset \pic^{-R_u(\pi^*(l'_d))}(\tX_u)$.

It means that for a generic choice of the contact points $p_1, \cdots, p_N$, the line bundle $\calL_u$ becomes a generic line bundle in $ r_s^{-1}(\calL_s)$.

We prove the following easy lemma:

\begin{lemma}\label{genericdisjoint}
For a generic line bundle $\calL_u \in  r_s^{-1}(\calL_s)$ the generic section in $H^0(\tX_u, \calL_u)$ has got  $-(l'_d, E_v)$ disjoint cuts along the exceptional divisor $E_v$.
\end{lemma}

\begin{proof}

Notice that a generic section $s \in H^0(\tX_s, \calL_s)$ has got a divisor which has got $-(l'_d, E_v)$ disjoint cuts along the exceptional divisor $E_v$. 

Consider the two dominant maps $H^0(\tX_s, \calL_s )_{reg, q} \to \eca^{- \pi^*(l'_{top}), \calL_s}(Z) \to  r_s^{-1}(\calL_s)$, 
where $H^0(\tX_s, \calL_s )_{reg, q}$ is the open set of regular sections, which do not vanish at the point $q$ (or at the points $q_1, \cdots, q_k$, if $t = 0$) and $Z$ is a very large cycle.
The first map simply takes a section to its divisor and the second map is the Abel map $c^{- \pi^*(l'_{top})}(Z) $ restricted to the subset $ \eca^{- \pi^*(l'_{top}), \calL_s}(Z) \subset \eca^{- \pi^*(l'_{top})}(Z)$, so the relative Abel map.

The first map is clearly surjective since every divisor $ D \in \eca^{- \pi^*(l'_{top}), \calL_s}(Z)$ determines the line bundle $\calL_s$ on the subsingularity $\tX_s$ and the second map is 
dominant since for a generic line bundle $\calL_u \in  r_s^{-1}(\calL_s)$ we have $H^0(\tX_u, \calL_u)_{reg} \neq \emptyset$.
Since the generic section in $H^0(\tX_s, \calL_s )_{reg, q}$ has got a divisor which has got $-(l'_d, E_v)$ disjoint arrows on the exceptional divisor $E_v$ it indeed follows that the generic section in $H^0(\tX_u, \calL_u)$ has got  $-(l'_d, E_v)$ disjoint arrows on the exceptional divisor $E_v$.

\end{proof}

Notice that we have the isomorphism $\pic^{-\pi^*(l'_d)}(\tX_{new}) \cong \pic^{-R_u(\pi^*(l'_d))}(\tX_u)$, and the line bundle $\calL_u$ does not have base points at the intersection of exceptional divisors. From this we get that the line bundle $\calO_{\tX_{new}}(- \pi^*(l'_{top})) $ has a section which has got $-(l'_d, E_v)$ dijsoint arrows on the exceptional divisor $E_v$. It means that the line bundle $\calL_u$ can have only simple base points on the exceptional divisor $E_v$ and this proves part 1) completely.

\bigskip

\textbf{Proof of part 2):}

\medskip

Assume first that $v \in \calv$, and there exists a cycle $A \geq E_v$, such that:

\begin{equation*}
\chi(l'_d) - h^1(\tX_1, \calL) + 1 = \chi(l'_d + A) - h^1(\tX_1, \calL(-A)).
\end{equation*}
We want to prove that there are $-(l'_d, E_v)$ distinct base points along $E_v$.

\medskip

Consider the the following exact sequence:

\begin{equation*}
0 \to H^0(\tX,  \calL_{d}(- E_v)) \to H^0(\tX, \calL_{d}) \to H^0(E_v, \calL_{d}  ) \to  H^1(\tX, \calL_{d}(- E_v))  \to H^1(\tX, \calL_{d}) \to 0.
\end{equation*}

Notice that $h^1(\tX, \calL_{d}) = h^1(\tX_1, \calL)$ and we have the following by Theorem\ref{relgen}:

\begin{equation*}
h^1(\tX, \calL_{d}( - E_v)) = \chi(l'_d + E_v) - \min_{0 \leq l \in L_{\mathcal{T}}}( \chi(l'_d + E_v + l) - h^1(\tX_1, \calL(-l - E_v) ). 
\end{equation*}

\begin{equation*}
h^1(\tX, \calL_{d}(- E_v)) = \chi(l'_d + E_v) - ( \chi(l'_d) - h^1(\tX_1, \calL) + 1)  = -(l'_d, E_v)  + h^1(\tX_1, \calL).
\end{equation*}

This means that the dimension of the image of the map $H^0(\tX, \calL_{m}) \to H^0(E_v, \calL_{d}  )$ is $1$, and since there are only simple base points on $E_v$ by part 1), we get that there $-(l'_d, E_v)$ disjoint base points along $E_v$.

\bigskip

Assume in the following that $v \in \calv$ and for every integer cycle $A \geq E_v$, one has:

\begin{equation*}
\chi(l'_d) - h^1(\tX_1, \calL) + 1 < \chi(l'_d + A) - h^1(\tX_1, \calL(-A)).
\end{equation*}

\medskip
By the same calculation as before we immediately get that the dimension of the image of the map $H^0(\tX, \calL_{d}) \to H^0(E_v, \calL_{d})$ is more than $1$.
It means that there can't be $-(l'_d, E_v)$ disjoint base points along the exceptional divisor $E_v$.

We know that $p_g(\tX) > p_g(\tX_1)$, which means similarly as in the proof of part 1) that there is a vertex $u \in \calv \setminus \calv_1$, such that there are differential forms in $H^1(\calO_{\tX})^* \cong H^0(\tX\setminus E, \Omega^2_{\tX})/ H^0(\tX,\Omega_{\tX}^2)$, which have got pole on the exceptional divisor $E_u$.

Blow up $E_u$ in generic points sequentially and let the new vertices be $u_0 = u, u_1, \cdots, u_t$ and let $t$ be the maximal positive integer, such that there are nonzero differential forms in $H^1(\calO_{\tX})^* \cong H^0(\tX\setminus E, \Omega^2_{\tX})/ H^0(\tX,\Omega_{\tX}^2)$ having pole on $E_{u_t}$.
Denote the new singularities by $\tX_{b, 1} = \tX_1, \tX'_b, \tX_b$ and the subsingularity with resolution graph $\mathcal{T} \setminus u_t \cup u_1, \cdots, u_{t-1}$ by $\tX_s$,  notice that in the case $t = 0$ the singularity $\tX_s$ is not nessecarily connected.
Notice that $\tX_s$ is a relatively generic singularity corresponding to $\tX_1$ and the cuts and $\tX_b$ is a relatively generic singularity corresponding to $\tX_s$ (and the corresponding cuts).

We know that $p_g (\tX_s) - p_g(\tX_1) < p_g (\tX) - p_g(\tX_1)$ and $\tX_s$ is relatively generic to $\tX_1$, which means that we can use our statements for the pair $(\tX_s, \tX_1)$.
Let us denote $ \calL_s = \pi^*(\calL_d) | \tX_s$.

From the fact that the dimension of the image of the map $H^0(\tX, \calL_{d}) \to H^0(E_v, \calL_{d})$ is more than $1$ we get immediately that the dimension of the image of the map $H^0(\tX_s, \calL_{s}) \to H^0(E_v, \calL_{s})$ is more than $1$.
We can use the induction hipothesis for the pair $(\tX_s, \tX_1)$, which means that the line bundle $\calL_s$ does not have a base point along the exceptional divisor $E_v$. 
If $\calL_s$ has some base points on $ \tX_s$ outside of $E_v$, then blow them up.

\medskip

\emph{This means that we reduced the statement of part 2) to the case, when $|\calv \setminus \calv_1| = 1$, a vertex called $u$, and the differential forms in $ H^0(\tX\setminus E, \Omega^2_{\tX})/ H^0(\tX,\Omega_{\tX}^2)$ have a pole along the exceptional divisor $E_u$ with order at most $1$.}

\medskip

We have to prove that if $v \in \calv$ arbitrary and the dimension of the image of the map $H^0(\tX, \calL_{d}) \to H^0(E_v, \calL_{d} )$ is more than $1$, then $\calL_{d}$ does not have a base point along the exceptional divisor $E_v$.
Let $p \in E_v$ be a generic point and blow up the exceptional divisor $E_v$ in $p$.
Denote the new exceptional divisor by $E_{new}$ and the new singularities by $\tX_{1, new}, \tX_{new}, \tX'_{new}$.

\medskip

With this setup we prove first the next key lemma:

\begin{lemma}\label{sreldom}
1) Assume that $v \in \calv_1$ and consider the line bundle $\calL_{new} = \pi^*(\calL)(- E_{new})$ on $\tX_{1, new}$, we claim that $(-\pi^*(l'_d)- E_{new}, \calL_{new})$ is relatively dominant on $\tX_{new}$.
\medskip

2) Assume that $v=u$, we claim that $(-\pi^*(l'_d)- E_{new}, \calL)$ is relatively dominant on $\tX_{new}$.

\end{lemma}
\begin{proof}

\textbf{For part 1)} by Theorem\ref{th:dominantrel} we have to prove that for all cycles $0 < \pi^*(l) + a E_{new} \in L_{\mathcal{T}_{new}}$, such that $\pi^*(l'_d) + (a+1) E_{new} + \pi^*(l) \in S'_{\mathcal{T}_{new}}$ and $H^0(\tX_{1, new}, \pi^*(\calL)(-\pi^*(l) - (a+ 1)E_{new}))_{reg} \neq \emptyset$ one has:

\begin{equation*}
\chi(\pi^*(l'_d) + E_{new}) - h^1(\tX_{1, new}, \calL_{new}) < \chi(\pi^*(l'_d) + (a+1) E_{new} + \pi^*(l)) - h^1(\tX_{1, new}, \pi^*(\calL)(-\pi^*(l) - (a+ 1)E_{new})).
\end{equation*}

We know that $h^1(\tX_{1, new}, \calL_{new}) = h^1(\tX_{1}, \calL)$ since the line bundle $\calL$ does not have a base point at the point $p$.
This means we have to prove:

\begin{equation*}
\chi(l'_d)  + 1 - h^1(\tX_{1}, \calL) < \chi(l'_d + l) + \frac{(a+1)(a+2)}{2} - h^1(\tX_{1, new}, \pi^*(\calL) (-\pi^*(l) - (a+ 1)E_{new}) ).
\end{equation*}

\medskip

\underline{Assume first that $l >0$:} 

\medskip

In this case we know that
\begin{equation*}
\chi(l'_d)  + 1 - h^1(\tX_{1}, \calL) \leq \chi(l'_d + l)  - h^1(\tX_1, \calL(-l)).
\end{equation*}

It means that we only have to prove:

\begin{equation*}
\frac{(a+1)(a+2)}{2} - h^1(\tX_{1, new}, \pi^*(\calL) (-\pi^*(l) - (a+ 1)E_{new}) ) > - h^1(\tX_1, \calL(-l)).
\end{equation*}

However it immediately follows from $H^0(\tX_1, \calL(- l))_{reg} \neq \emptyset$, $H^0(\tX_{1, new}, \pi^*(\calL) (-\pi^*(l) - (a+ 1)E_{new}))  \subset H^0(\tX_1, \calL(-l))$ and
$H^0(\tX_{1, new}, \pi^*(\calL) (-\pi^*(l) - (a+ 1)E_{new}))  \neq H^0(\tX_1, \calL(-l))$.
The second strict containment follows from the fact that the line bundle $ \calL(-l)$ does not have a base point at $p$, since $p$ is a generic point on $E_v$.

\underline{Assume that $l = 0$:}

\medskip

In this case we have $a > 0$ and we have to prove:

\begin{equation*}
1 - h^1(\tX_{1, new}, \pi^*(\calL)(- E_{new})) <  \frac{(a+1)(a+2)}{2} - h^1(\tX_{1, new}, \pi^*(\calL)(- (a+ 1)E_{new})).
\end{equation*}

For this we only have to prove that $ H^0(\tX_{1, new}, \pi^*(\calL)(- E_{new})) \neq H^0(\tX_{1, new}, \pi^*(\calL)(-  (a+ 1) E_{new}))$. This immediately follows from the fact that $\calL$ does not have a base point on $E_v$, which means that the generic section $s \in H^0(\tX_1, \calL)$ has got $-(l'_d, E_v)$ disjoint arrows on $E_v$.

\bigskip

\textbf{For part 2)} we have to prove that for all cycles $0 < \pi^*(l) + a E_{new} \in L_{\mathcal{T}_{new}}$, such that $\pi^*(l'_d) + (a+1) E_{new} + \pi^*(l) \in S'_{\mathcal{T}_{new}}$ 
and $H^0(\tX_1, \calL(-\pi^*(l) - (a+ 1)E_{new}) )_{reg} \neq \emptyset$ one has:

\begin{equation*}
\chi(\pi^*(l'_d) + E_{new}) - h^1(\tX_{1}, \calL) < \chi(\pi^*(l'_d) + (a+1) E_{new} + \pi^*(l)) - h^1(\tX_1, \calL(-\pi^*(l) - (a+ 1)E_{new}) ).
\end{equation*}

\begin{equation*}
\chi(l'_d)  + 1 - h^1(\tX_{1}, \calL) < \chi(l'_d + l)  +  \frac{(a+1)(a+2)}{2} - h^1(\tX_1, \calL(-l)).
\end{equation*}

\medskip

\underline{Assume first that $l_v >0$:}

\medskip

In this case we know that
\begin{equation*}
\chi(l'_d)  + 1 - h^1(\tX_{1}, \calL) < \chi(l'_d + l)  - h^1(\tX_1, \calL(-l)).
\end{equation*}

The statement immediately follows in this case, since $\frac{(a+1)(a+2)}{2} \geq 0$.

\medskip

\underline{Assume that $l >0$, but $l_v = 0$:}

\medskip

In this case we have $a \geq 0$, so $\frac{(a+1)(a+2)}{2} \geq 1$, and $\chi(l'_d)  + 1 - h^1(\tX_{1}, \calL) \leq \chi(l'_d + l)  - h^1(\tX_1, \calL( -l))$.
The statement follows again immediately.

\medskip

\underline{Assume finally that $l =0$:}

\medskip

In this case we have $a > 0$ and we have to prove:

\begin{equation*}
\chi(l'_d)  + 1 - h^1(\tX_{1}, \calL) < \chi(l'_d)  +  \frac{(a+1)(a+2)}{2} - h^1(\tX_1, \calL ).
\end{equation*}

Notice that this is also trivial, because if $a > 0$, then one has $\frac{(a+1)(a+2)}{2} > 1$.
\end{proof}

\medskip

In the following we continue to prove part 2) of our main theorem, assume first that $v \in \calv_1$, which is equivalent to $v \neq u$. 
\medskip

We know that there is a large positive integer $N$, such that $\dim(\im(c^{-N E_{u}^*}(Z))) = e_{\tX}(u)$, where $Z$ is a large cycle on $\tX$ and blow up $E_u$ in $N$ different generic points $p_1, \cdots p_N$ and let the new exceptional divisors be $E_{w_1}, \cdots, E_{w_N}$.
Denote the new singularites we get by $\tX_{new}, \tX'_{new}$ and denote its subsingularity with vertex set $\calv$ by $\tX_u$.
Similarly as before we get $p_g(\tX_u) = p_g(\tX)$.
We know that the differential forms in $H^0(\tX\setminus E, \Omega^2_{\tX_u})/ H^0(\tX,\Omega_{\tX_u}^2)$ have got a pole on the exceptional divisor $E_u$ of order 
at most $1$. This means that if $D'$ is a smooth divisor in $\tX_u$ which intersects the exceptional divisor $E_u$ at a regular point then the line bundle $\calO_{\tX_u}(D')$
depends only on the intersection point $D' \cap E_u$.
Denote $l'_u = R_u(\pi^*(l'_d))$, we know that $\pic^{-\pi^*(l'_d)}(\tX_{new}) \cong \pic^{-l'_u}(\tX_u)$ and the isomorphism map is given by the restriction map $r_u$.

As before we will again deform the relatively generic analytic structure $\tX'_{new}$ by changing the plumbing with the tubular neighborhoods of the exceptional divisors be $E_{w_1}, \cdots, E_{w_N}$ while moving the contact points $p_1, \cdots, p_N$:
If we do such a minor deformation then the analytic structure on $\tX'_{new}$ still remains relatively generic with respect to $\tX_1$ and the cuts.

We know that $\dim(\im(c^{-N E_{u}^*}))= e_{\tX}(u)$ and the coefficients of $E_{w_1}, \cdots, E_{w_N}$ in $ \pi^*(l'_{top})$ are positive. It means that exactly the same way as in Lemma \ref{coveropen} we get that  if we move the contact points $p_1, \cdots, p_N$, then the restricted line bundle $\calL_u = r_u( \calO_{\tX_{new}}(- \pi^*(l'_{top})) ) $ cover an open set in $ r^{-1}(\calL) \subset \pic^{-l'_u}(\tX_u)$.
It means that for a generic choice of the contact points $p_1, \cdots, p_N$ the line bundle $\calL_u$ is a generic line bundle in $ r^{-1}(\calL)$.

\medskip

This means that it is enough to prove the following lemma:

\medskip

\begin{lemma}\label{gennobase}
A generic line bundle $\calL_u \in  r^{-1}(\calL)$ does not have a base point on $E_v$.
\end{lemma}
\begin{proof}

Consider a large cycle $Z$ on $\tX_u$ and  look at the space $\eca^{ -l'_u, \calL}(Z) \subset \eca^{-l'_u}(Z)$ consisting of divisors $D \in \eca^{-l'_u}(Z)$, such that $r(D)$ gives the line bundle $\calL$ on $\tX_1$, this is a smooth and irreducible by Corollary \ref{cor:smoothirreddim}.
Assume to the contrary that a generic line bundle $\calL_u \in  r^{-1}(\calL)$ has a base point on the exceptional divisor $E_v$.
Since the relative Abel map $\eca^{ -l'_u, \calL}(Z)  \to  r^{-1}(\calL)$ is dominant and $\eca^{ -l'_u, \calL}(Z)$ is irreducible, this means that there is open subset $U \in \eca^{ -l'_u, \calL}(Z)$ and a map $ f: U \to E_v$, such that $f(D) \in |D|$ is a base point of the line bundle $\calO_{\tX_u}(D)$.
Consider a generic divisor $D \in U$ which has got $-(l'_d, E_v)$ disjoint arrows on the regular part of $E_v$ and blow up $E_v$ at the generic point $f(D)\in E_v$, let the new exceptional divisor be $E_{f(D)}$. Denote the new singularities by$\tX_{1, b}, \tX_{b}, \tX'_b$ and the line bundle $\calL_b = \pi^*(\calL) \otimes \calO_{\tX_{1, b}}( - E_{f(D)})$, and denote $Z_b = Z + (Z_v -1) E_{f(D)}$.

\medskip

We know that $ \eca^{-l'_u, \calL}(Z) \cap \eca^{-\pi^*(l'_u) -E_{f(D)}}(Z_b) = \eca^{-\pi^*(l'_u) -E_{f(D)}, \calL_b}(Z_b)$.
We can also assume that $D$ is enough generic, such that the intersection $ \eca^{-l'_u, \calL}(Z) \cap \eca^{-\pi^*(l'_u) -E_{f(D)}}(Z_b)$ is transvesal in $D$.
If we consider a small open neighborhood $U_1 \subset \eca^{-\pi^*(l'_u) -E_{f(D)}, \calL_b}(Z_b)$ of $D$, then we get by the continuity of the map $f$ that if $D' \in U_1$, then $f(D') = f(D)$.

This means that $f(D) = f(D')$ is a base point of the line bundle $\calO_{\tX_u}(D')$, which means that $$h^1(\calO_{\tX_b}(D')) = h^1(\calO_{\tX_u}(D')) + 1 = h^1(\tX_1, \calL) + 1.$$

From this we get that the map $\eca^{-\pi^*(l'_u) -E_{f(D)}, \calL_b}(Z_b) \to r^{-1}(\calL_b) \subset \pic^{-\pi^*(l'_u) -E_{f(D)}}(\tX_b)$ cannot be a submersion in any of the points $D' \in U_1$, because otherwise we would have $h^1(\calO_{\tX_b}(D')) = h^1(\tX_{1, b}, \calL_b) =  h^1(\tX_1, \calL)$.

\medskip

Indeed we know by \cite{NNA1} that $h^1(\calO_{\tX_b}(D'))$ equals the codimension of the image of the tangent map $T_{D'}(c^{-\pi^*(l'_u) -E_{f(D)}}(Z_b))$ and
$h^1(\tX_{1, b}, \calL_b)$ equals the codimension of the image of the tangent map $T_{r(D')}(c^{-R(\pi^*(l'_u) -E_{f(D)})}(Z_{1, b}))$.
On the other hand if the map $\eca^{-\pi^*(l'_u) -E_{f(D)}, \calL_b}(Z_b) \to r^{-1}(\calL_b) \subset \pic^{-\pi^*(l'_u) -E_{f(D)}}(\tX_b)$ would be a submersion in the point $D' $
then the image of the tangent map $T_{D'}(c^{-\pi^*(l'_u) -E_{f(D)}}(Z_b))$ would contain the tangent space of $r^{-1}(\calL_b)$.
We also know that the projection of the image of the tangent map $T_{D'}(c^{-\pi^*(l'_u) -E_{f(D)}}(Z_b))$ is exaxtly the  image of the tangent map $T_{r(D')}(c^{-R(\pi^*(l'_u) -E_{f(D)})}(Z_{1, b}))$, so we would indeed get $h^1(\calO_{\tX_b}(D')) = h^1(\tX_{1, b}, \calL_b)$, which is impossible.

On the other hand by part 1) of Lemma\ref{sreldom} we know that the map $\eca^{-\pi^*(l'_u) -E_{f(D)}, \calL_b}(Z_b) \to r^{-1}(\calL_b)$ is dominant, and this is a contradiction, since $U_1 \subset \eca^{-\pi^*(l'_u) -E_{f(D)}, \calL_b}(Z_b)$ is open.
\end{proof}

If $v = u$ the proof is the very same if one use part 2) of Lemma\ref{sreldom}, just even simpler. It finishes the proof of the theorem completely.
\end{proof}

\section{Multiplicity of relatively generic surface singularities}

In this section we aim to ivestigate the multiplicity of a relatively generic singularity $\tX$.
In order to determine the multiplicity of the singularity we need to know the base point structure of the line bundle $\calO_{\tX}(-Z_{max})$.
From Theorem\ref{basepointrel} we know how many simple base points the line bundle $\calO_{\tX}(- Z_{max})$ has on the regular part of an exceptional divisor $E_v$. On the other hand for the computation of the multiplicity of the singularity we also need to know for the base points that for which number $t$ they are $t$-simple.

\medskip

Consider two rational homology sphere resolution graphs $\mathcal{T}_1, \mathcal{T}$, a singularity $\tX_1$ with resolution graph $\mathcal{T}_1$ and cuts $D_{v_2}$.
Assume that $\tX$ is a relatively generic singularity with resolution graph $\mathcal{T}$ with respect to the subsingularity $\tX_1$ (and cuts $D_{v_2}$ on it). 
Let $Z_{max}$ be the maximal ideal cycle of $\tX$. Assume that the line bundle $\calO_{\tX_1}(-Z_{max})$ does not have a base point.

\medskip

We have the following proposition with this setup:

\begin{proposition}\label{tsimple}

Assume that $v \in \calv$ is a vertex such that the line bundle $\calO_{\tX}(- Z_{max})$ has a base point along $E_v$. Let $t \geq 1$ be the maximal integer such that there exists a cycle $A \geq t \cdot E_v$ for which:

\begin{equation*}
\chi(Z_{max}) - h^1(\calO_{\tX_1}(- Z_{max}))+ 1  = \chi(Z_{max} + A) - h^1(\calO_{\tX_1}(- Z_{max} - A)). 
\end{equation*}
\medskip

Let $p$ be a base point of the line bundle $\calO_{\tX}(- Z_{max})$ on the exceptional divisor $E_v$ and let $t_p \geq 1$ the integer for which $p$ is $t_p$-simple, then we have $t_p \geq t$.
\end{proposition}
\begin{proof}

Blow up $E_v$ sequentially, and let the new exceptional divisors be $E_{v_1}, \cdots E_{v_i}$. Denote the singularity we get after the $i$-th blowup by $\tX_i$ and the
blow up map by $\pi_i^* : \tX_i \to \tX$.
At the first step we blow up $E_v$ at the base point $p$.
At the $i$-th step if the line bundle $ \calO_{\tX_i}(- \pi_i^*(Z_{max}) - \sum_{1 \leq j \leq i}j E_{v_j} )$ has a base point along the divisor $E_{v_i}$, then we blow it up again at the base point.
Let us denote the effective integer cycle $ \sum_{1 \leq j \leq i}j E_{v_j} = A_i$.
\medskip

Assume that $i \leq t-1$, we will prove that $\calO_{\tX_i}(- \pi_i^*(Z_{max}) - A_i)$ has a base point on the exceptional divisor $E_{v_i}$, from the alternative definition of $t_p$ we indeed get that $t_p \geq t$.

We will prove it by induction on $i$, if $i = 0$ this follows from our main Theorem\ref{basepointrel}.
Assume that the statement holds for $0, \cdots, i-1$ and notice that this means  $h^1(\calO_{\tX_i}(- \pi_i^*(Z_{max}) - A_i)) = i + h^1(\calO_{\tX}(- Z_{max}))$.

\medskip

Consider the following exact sequence, where we denote $\calL_i=  \calO_{\tX_i}(- \pi_i^*(Z_{max}))$ :

\begin{equation*}
H^0(\tX_{i},  \calL_i (-A_i)) \to  H^0(E_i,  \calL_i(-A_i)) \to H^1(\tX_{i}, \calL_i( -A_i - E_i)) \to  H^1(\tX_{i},  \calL_i(-A_i)) \to 0
\end{equation*}

Notice first that $ h^0(E_i,  \calL_i (-A_i)) = 2$ and $h^1(\tX_{i},  \calL_i(-A_i)) = i + h^1(\calO_{\tX}(- Z_{max}))$.

It means that we only have to prove $ h^1(\tX_{i}, \calL_i (-A_i - E_i)) \geq i+1 + h^1(\calO_{\tX}(- Z_{max}))$.

\medskip

We know that there is a cycle $A \geq t E_v $ such that:

\begin{equation*}
\chi(Z_{max}) - h^1(\calO_{\tX_1}(- Z_{max}))+ 1  = \chi(Z_{max} + A) - h^1(\calO_{\tX_1}(- Z_{max} - A)).
\end{equation*}

It means by Theorem \ref{relgen} that $\dim \left( \frac{H^0(\calO_{\tX}(- Z_{max}))}{H^0(\calO_{\tX}(- Z_{max} - A))} \right)= 1$.

Notice that $H^0(\calO_{\tX_i}(- \pi_i^*(Z_{max} + A)) ) \subset H^0(\tX_{i}, \calL_i(-A_i - E_i))$, which means that:

\begin{equation*}
h^1(\tX_{i}, \calL_i(-A_i - E_i)) \geq h^1( \calO_{\tX_i}(- \pi_i^*(Z_{max} + A))) +  \chi(\pi_i^*(Z_{max}) + A_i + E_i ) - \chi( \pi_i^*(Z_{max}+ A)).
\end{equation*}

\begin{equation*}
h^1(\tX_{i}, \calL_i(-A_i - E_i)) \geq h^1(\calO_{\tX}(- Z_{max} - A)) +  \chi(Z_{max}) + (i+2)  - \chi(Z_{max} + A).
\end{equation*}

\medskip

We also know that $h^1(\calO_{\tX}(- Z_{max})) = h^1(\calO_{\tX}(- Z_{max} - A)) + \chi(Z_{max}) - \chi(Z_{max} + A) + 1$.
We get from this that $h^1(\tX_{i}, \calL_i(-A_i - E_i)) \geq i+1 + h^1(\calO_{\tX}(- Z_{max}))$ and the proposition is proved.

\end{proof}

\begin{proof}[of Theorem \textbf{D}]
Part 1) follows from the fact that $Mult(\tX) = -Z_{max}^2 + \sum_{p \in B} t_p$ where $B$ is the set of base points and $t_p \geq t_v$ if $p \in E_v$.
Under the conditions of part 2) the line bundle $\calO_{\tX}(- Z_{max})$ does not have any base points from our main Theorem\ref{basepointrel}, so indeed we get $Mult(\tX) = -Z_{max}^2$.
\end{proof}

\begin{remark}
Notice that if the line bundle $\calO_{\tX_1}(-Z_{max})$ has a base point, we can use Theorem \textbf{C} or Theorem \textbf{D} in the way that we blow up its base points and get a situation when
the line bundle $\calO_{\tX_{1, new}}(- Z_{max, new})$ does not have a base point.
After that we can use Theorem \textbf{C} or Theorem \textbf{D} to get information about the base point structure of the line bundle $\calO_{\tX_{new}}(- Z_{max, new})$ or on the multiplicity
of the singularity $\tX$.
\end{remark}

\section{Corollaries of the main results}

\subsection{Relatively rational resolution graphs}

In this subsection we will define and characterise the natural analouge of rational singularities in the case of the relative setup.
Let us recall that a singularity $\tX$ with resolution graph $\mathcal{T}$ is called rational if $p_g(\tX) = 0$ and this property depends just on the topological type $\mathcal{T}$
being equivalent to $\min_{0 < l \in L} \chi(l) > 0$.

\medskip

In the following we prove an analouge theorem in the relative case (generalised for cycles), we fix some notations first.

Consider two rational homology sphere resolution graphs $\mathcal{T}_1 \subset \mathcal{T} $ with vertex sets $\calv_1 \subset \calv$, where $\calv = \calv_1 \cup \calv_2$  and a fixed singularity $\tX_1$ for the resolution graph $\mathcal{T}_1$, and cuts $D_{v_2}$ (for vertices $v_2$ which have a neighbour in $\calv_1$).

Assume furthermore that we have an effective cycle $Z$ on $\mathcal{T}$ and write $Z = Z_1 + Z_2$, where $|Z_1| \subset \calv_1$ and $|Z_2| \subset \calv_2$.

Consider a singularity with resolution $\tX$ and resolution graph $\mathcal{T}$, which has the subsingularity $\tX_1$ and is compatible with the cuts $D_{v_2}$ ($E_{v_2} \cap \tX_1 = 
D_{v_2}$ for vertices $v_2$ which have a neighbour in $\calv_1$).

With these notations we have the following theorem:

\begin{corollary}\label{relativerational}

We have $h^1(\calO_Z) = h^1(\calO_{Z_1})$ if and only if $(0, \calO_{Z_1})$ is relative dominant on the cycle $Z$, or equivalently:

\begin{equation*}
- h^1(\calO_{Z_1}) < \chi(l)-  h^1((Z-l)_1, \calO_{(Z-l)_1}(-l)),
\end{equation*}
for all $0 < l \leq Z$.

In particular this property is independent of the chosen analytic type $\tX$.
\end{corollary}

\begin{proof}

Assume first that $h^1(\calO_Z) = h^1(\calO_{Z_1})$ for some fixed analytical structure $\tX$ on $\mathcal{T}$.

In this case $(0, \calO_{Z_1})$ is relative dominant on the cycle $Z$, because we have obviously $\eca^{0, \calO_{Z_1}}(Z) \neq \emptyset$ (the empty divisor is in $\eca^{0, \calO_{Z_1}}(Z)$). On the other hand $\dim(r^{-1}( \calO_{Z_1})) = 0$, where $r$ is the restriction map $\pic^{0}(Z) \to \pic^{0}(Z_1)$ so the map $\eca^{0, \calO_{Z_1}}(Z) \to r^{-1}( \calO_{Z_1})$ is indeed dominant.

\medskip

In the other direction assume that  $(0, \calO_{Z_1})$ is relative dominant on the cycle $Z$.
Notice that $\dim(\eca^{0, \calO_{Z_1}}(Z)) = h^1(\calO_{Z_1}) - h^1(\calO_{Z_1}) + (0, Z) = 0$, and we know that the map $ \eca^{0, \calO_{Z_1}}(Z) \to r^{-1}( \calO_{Z_1})$ is dominant. We get from this that $\dim(r^{-1}( \calO_{Z_1})) = 0$, which means $h^1(\calO_Z) = h^1(\calO_{Z_1})$ and it proves the theorem completely.
\end{proof}

If we apply Corollary \ref{relativerational} in case of a very large cycle $Z$, then by the formal neighborhood theorem we get that the property $p_g(\tX) = p_g(\tX_1)$ is independent
of the analytic type of $\tX$.
It holds if and only if:

\begin{equation*}
- p_g(\tX_1) < \chi(l)-  h^1(\calO_{\tX_1}(-l)),
\end{equation*}
for all $0 < l \in L$.

\begin{defn}
If the property above holds, then we call the resolution graph $\mathcal{T}$ relatively rational to the singularity $\tX_1$ and the cuts $D_{v_2}$ on it.
\end{defn}

\subsection{Possible geometric genera of elliptic graphs}

In the following as another corollary we reprove the classification of the possible geometric genuses for a minimal, rational homology sphere elliptic resolution graph $\mathcal{T}$.
We will prove in some sense this theorem in much higher generality in a following manuscript, namely we will prove that the possible geometric genuses for a rational homology sphere resolution graph is an interval of integers.

Notice that from \cite{NNA3} we know that such a rational homology sphere elliptic graph contains a largest numerically Gorenstein subgraph $\mathcal{T}'$, such that if there are two singularities $\tX' \subset \tX$ corresponding to them, then $p_g(\tX') = p_g(\tX)$.

It means that we can asume in the following that the resolution graph $\mathcal{T}$ is numerically Gorenstein, and consider its elliptic sequence $\calv = B_0, B_1, \cdots B_m$.
Let us recall that if $0 \leq i \leq m$, then $B_i$ is a numerically Gorenstein elliptic subgraph, whose elliptic sequence is $B_i, B_{i+1}, \cdots, B_m$ and it's maximal geometric genus is $m-i+1$ by
\cite{weakly}. Furthermore also by \cite{weakly}, if there is a singularity $\tX_i$ with resolution graph $\mathcal{T}_i$ and $p_g(\tX_i) = m-i+1$, with subsingularities $\tX_{i+1}, \cdots, \tX_m$, then $p_g(\tX_j) = m-j+1$, where $i \leq j \leq m$ and $\tX_i$ is Gorenstein.

We have the following theorem with this setup:

\begin{theorem}
Let $\mathcal{T}$ be a minimal, elliptic, rational homology sphere, numerically Gorenstein resolution graph with elliptic sequence $\calv = B_0, B_1, \cdots B_m$, then the possible geometric genera of normal surface singularities with resolution graphs $\mathcal{T}$ are $1, 2, \cdots, m+ 1$.

More presicely, if $0 \leq i \leq m$ and $\tX_i$ is a Gorenstein singularity with resolution graph $B_i$ such that $p_g(\tX_i) = m-i+1$ and $\tX$ is a relatively generic singularity with resolution graph $\mathcal{T}$ corresponding to the subsingularity $\tX_i$ and generic cuts, then $p_g(\tX) = m-i+1$.
\end{theorem}
\begin{proof}

Consider a relatively generic singularity $\tX$ corresponding to the subsingularity $\tX_i$ and generic cuts on it.

By Theorem \textbf{B} we know that $p_g(\tX) = 1 - \min_{l \geq E}(\chi(l) - h^1(\calO_{\tX_i}(-l))$.
By Remark\ref{eleg}, in the minimum we are enough to take the cycles $l \geq E,  l \in S'$ such that $H^0( \calO_{\tX_i}(-l) )_{reg} \neq \emptyset$.
Assume in the following that $p_g(\tX) = 1 - \chi(l) +  h^1(\calO_{\tX_i}(-l))$, where $l \geq E, l \in S'$ and $H^0(\calO_{\tX_i}(-l))_{reg} \neq \emptyset$.

\medskip
Notice that $R(l) \in S'(\mathcal{T}_i)$,we claim first that $R(l) \neq 0$:

Indeed, if $R(l) = 0$ , then $\calO_{\tX_i}(-l) = $ should be the trivial line bundle. On the other hand we know that $l \geq E$, and there is a vertex $v \in \calv$, which has got a neighbour $w$ in $\mathcal{T}_i$. We know that $v \in |l|$, but we glue the exceptional divisor $E_v$ along a generic cut along $E_w$.
This is impossible, because $\tX_i$ is numerically Gorenstein, so we have $e_{\tX}(v) > 0$ from \cite{NNA3}. This means that if we deform the cut along we glue $E_w$, then the line bundle
$ \calO_{\tX_i}(-l) $ changes, so for a generic choice this line bundle is nontrivial, which is a contradiction.

It means that we have $R(l) \in S'(\mathcal{T}_i) \setminus 0$.
Notice that $\calO_{\tX_i}(-l)  \in \im(c^{-R(l)})$ and since $\tX_i$ is numerically Gorenstein we have $\dim( \im(c^{-R(l)})) \geq 1$.
We know from \cite{NNA3} that $h^1$ is uniform on $\overline{ \im(c^{-R(l)})}$ and it is $p_g(\tX_i) - \dim( \im(c^{-R(l)})  )$ which means that $h^1(\calO_{\tX_i}(-l) ) \leq m-i$.

\medskip

It means that we have $p_g(\tX) \leq 1- \chi(l) + m-i \leq m-i+1$.
On the other hand we know that $p_g(\tX) \geq p_g(\tX_i) = m-i+1$, which means that $p_g(\tX) = m-i+1$ and we are done.
\end{proof}

\end{document}